\documentclass[a4paper]{article}
\usepackage[affil-it]{authblk}
\usepackage[authoryear]{natbib}
\usepackage{orcidlink}
\newenvironment{frontmatter}{}{}

\let\address\affil
\newenvironment{keyword}{\small \textbf{Keywords: }}{}

\usepackage{amssymb, amsmath, amsfonts,amsthm}
\usepackage{graphicx}
\usepackage{xcolor}
\usepackage{esvect}
\usepackage{hyperref}
\usepackage{todonotes}
\usepackage{xspace}
\usepackage{subcaption}
\usepackage{multirow}
\usepackage[lined,algoruled,commentsnumbered,noend]{algorithm2e}
\usepackage{marvosym}%for corresponding author letter
\usepackage[normalem]{ulem} %for \sout
\usepackage{graphicx} % Required for inserting images
\usepackage{fullpage}

\newcommand{\Z}{\mathbb{Z}}
\newcommand{\R}{\mathbb{R}}

\DeclareMathOperator{\st}{\mbox{s.t. }}

\DeclareMathOperator{\intset}{int}

\newcommand{\cD}{\mathcal{D}}
\newcommand{\cP}{\mathcal{P}}
\newtheorem{Theorem}{Theorem}
\newtheorem{Assumption}[Theorem]{Assumption}

\newtheorem{Lemma}[Theorem]{Lemma}
\newtheorem{Proposition}[Theorem]{Proposition}
\newtheorem{Corollary}[Theorem]{Corollary}

\newcommand{\VFR}{VFR\xspace}
\newcommand{\SLR}{SLR\xspace}

\newcommand{\BC}{B\&C\xspace}

\def\yy{\hat{y}}

\newcommand{\DCs}{DCs\xspace}
\newcommand{\DC}{DC\xspace}

\newcommand{\textoverline}[1]{$\overline{\mbox{#1}}$}
\newcommand{\ccrSLR}{\textoverline{\SLR}\xspace}

\makeatletter
\renewcommand*{\top}{%
{\mathpalette\@transpose{}}%
}
\newcommand*{\@transpose}[2]{%
\raisebox{\depth}{$\m@th#1\scriptscriptstyle\mathsf{T}$}%
}
\makeatother

\newcommand{\Set}[1]{\left\{#1\right\}}

\newcommand{\Defset}[3][\defsep]{\Set{#2#1#3}}

\begin{document}
\begin{frontmatter}

\title{On solving integer bilevel optimization problems with a non-convex quadratic follower objective function using disjunctive cuts}

\makeatletter
\def\blfootnote{\xdef\@thefnmark{}\@footnotetext}
\makeatother
\blfootnote{
		\textit{Email addresses:} 
		\texttt{ktaninmis@ku.edu.tr} (Kübra Tanınmış),
        \texttt{elisabeth.gaar@uni-a.de} (Elisabeth Gaar),
        \texttt{jonxlee@umich.edu} (Jon Lee),
        \texttt{ivana.ljubic@wu.ac.at} (Ivana Ljubi\'c), 
        \texttt{markus.sinnl@jku.at} (Markus Sinnl)
        }
        
\author[1]{K\"ubra Tan{\i}nm{\i}\c{s}\,\orcidlink{0000-0003-1081-4182}}
\author[2]{Elisabeth Gaar\,\orcidlink{0000-0002-1643-6066}}
\author[3]{Jon Lee\,\orcidlink{0000-0002-8190-1091}}
\author[4]{Ivana Ljubi\'c\,\orcidlink{0000-0002-4834-6284}}
\author[5]{Markus Sinnl\,\orcidlink{0000-0003-1439-8702}}

\address[1]{\small Department of Industrial Engineering, Ko\c{c} University, \.{I}stanbul, Turkey}
\address[2]{Institute of Mathematics, University of Augsburg, Augsburg, 
Germany}
\address[3]{University of Michigan, Ann Arbor, Michigan, USA}
\address[4]{Institute for Statistics and Mathematics, Vienna University of Economics and Business,  Vienna, Austria}
\address[5]{Institute of Business Analytics and Technology Transformation/JKU Business School, Johannes Kepler University Linz, Linz, Austria}
\date{}

\maketitle

\begin{abstract}
In this work, we study bilevel optimization problems where all variables are integer, all constraints and the leader objective function are linear, and the follower objective function is non-convex quadratic. Relying on bilevel-free sets derived from improving directions, we develop a disjunctive cut approach to exclude bilevel-infeasible solutions within a branch-and-cut algorithm. We show that our disjunctive cuts can be obtained by solving a cut generating linear program. 
Furthermore,  we discuss conditions that allow the number of disjuncts in the cut generating linear program to be reduced, and we propose several strategies to identify improving directions and generate disjunctive cuts efficiently. 

We evaluate various aspects of the proposed branch-and-cut algorithm on both convex instances from the literature that fit our setting and new non-convex instances and compare the performance of our best approach with existing state-of-the-art approaches, which we significantly outperform. 
\end{abstract}


\end{frontmatter}

%\date{May 2023}

%%%%%%%%%%%%%

\section{Introduction}
Bilevel optimization problems (BOPs) arise naturally in settings where two decision makers act sequentially and hierarchically, with the so-called \emph{leader} making decisions in the \emph{upper-level} in anticipation of a \emph{follower's} optimal response in the \emph{lower-level}. Such structures appear across a wide range of application domains, including energy market design, network interdiction, transportation planning, or pricing; we refer to recent surveys by~\citet{kleinert2021survey,Beck-et-al:2022,Caselli-et-al:2026} for introductions to the field.
From a computational standpoint, BOPs are notoriously difficult. Already the linear-linear case, where both the leader and follower problem are continuous linear programs (LPs), is strongly NP-hard~\citep{Hansen-et-al:1992}.  The complexity increases fundamentally when convexity of the follower problem cannot be assumed.
A particularly well-studied class of BOPs beyond the convex setting is that of integer linear BOPs, where non-convexity at the follower problem enters through integrality requirements. These BOPs are $\Sigma_2^p$-hard in general~\citep{jeroslow1985polynomial}, and solution approaches typically rely on branch-and-cut frameworks that alternate between branching on integer variables and generating cutting planes to remove solutions that are integer-feasible but violate bilevel-feasibility. Starting from the foundational work of~\citet{denegre2009branch}, the field has developed rapidly, with contributions including the watermelon algorithm~\citep{watermelon}, intersection cuts~\citep{fischetti2016intersection,fischetti2017new,Battista-Ralphs:2025}, and various families of valid inequalities~\citep{Tahernejad-Ralphs:2025}, see also the recent book \citep{Beck-Ljubic-Schmidt:2026} for a general overview.

In this paper, we study a class of BOPs that has so far received comparatively little attention: problems where the leader problem is an integer LP (ILP) and the follower minimizes a quadratic objective function, which we do \emph{not} assume to be convex, over a set of integer variables subject to linear constraints, i.e., the follower problem is a quadratic ILP (QILP). 
We refer to this class as ILP-QILP BOPs. 
More precisely, we consider the following \emph{optimistic} BOP
\begin{subequations}
\label{bilevelProblem}
\begin{alignat}{3}
\min_{(x,y)\in \R^n} ~ c^\top x + d^\top y \\
\st  Mx+Ny &\geq h \label{MxNy}\\
x &\in \mathbb{Z}^{n_1} \\
y &\in \Omega(x),  
\end{alignat}
\end{subequations}
where $\Omega(x)$ is the set of globally 
  optimal solutions to the $x$-parameterized \emph{follower} (or \emph{lower-level}) problem
  \begin{subequations}
\label{llProblem}
\begin{alignat}{3} 
\min_{y\in \R^{n_2}} ~ q(y) \qquad  \\
\st  A x+ B y &\geq f \label{general}\\
 \ell \le y &\le u \label{ybounds} \\
 y &\in \mathbb{Z}^{n_2} 
\end{alignat}
\end{subequations}
and the parameters 
$c \in \R^{n_1}$, 
$d \in \R^{n_2}$, 
$M \in \R^{m_1 \times n_1}$, 
$N \in \R^{m_1 \times n_2}$, 
$h \in \R^{           m_1}$, 
$A \in \Z^{m_2 \times n_1}$, 
$B \in \Z^{m_2 \times n_2}$,  
$f \in \Z^{m_2 }$, 
$\ell,u\in\Z^{n_2}$, and
$n_1, n_2 \in \Z^+$ are given.
Here, the decision variables $x$ and $y$ are integer vectors of dimension $n_1$ and $n_2$, respectively, $n=n_1+n_2$ and $q(y)$ is a quadratic function of the form 
\[ q(y) = y^\top R y + g^\top y \]
with $R\in \mathbb{Z}^{n_2\times n_2}$, with assuming that $R$ is symmetric without loss of generality, and $g \in \mathbb{Z}^{n_2}$.
We refer to the constraints
\eqref{general} as \emph{lower-level constraints}, and the subset of~\eqref{general} in which some leader variable appears with a non-zero coefficient as \emph{linking constraints}. 
Note that any polyhedral constraints on $y$ within the follower problem can be represented within~\eqref{general} by setting the entries in the corresponding row of $A$ to all zeros. While \emph{bound constraints}~\eqref{ybounds} could also be included in~\eqref{general} we write them down explicitly, as they play a central role in the derivation of our \emph{disjunctive cuts} (DCs).
Likewise, all polyhedral constraints on $x$ are assumed to be included in~\eqref{MxNy} (in which case all entries in the corresponding rows of $N$ are equal to zero). Throughout the paper we will denote the $i$-th row of any matrix $Z$ as $Z_i$ and the entry of $Z$ in the $i$-th row and $j$-th column as $Z_{ij}$. 

The non-convexity of the quadratic follower objective function $q(y)$ of ILP-QILPs has significant consequences and brings an additional layer of difficulty. Even if the other substantial source of difficulty, the integrality requirements of the follower variables $y$, were present, the non-convexity rules out the use of KKT-based or duality-based single-level reformulations, since the KKT conditions are no longer sufficient for global optimality of the follower problem. It also means that, for a fixed leader decision, finding the optimal follower response is itself still an NP-hard problem, even without integer $y$. 

%These two sources of difficulty (combinatorial follower decisions together with a potentially non-convex objective function) make this problem class substantially harder than MILP-MILP from a computational point-of-view.

\subsection{Previous and related work}

Existing work on BOPs with non-convex integer follower problems is sparse. General-purpose frameworks such as those of~\citet{Mitsos:2010} and~\citet{Kleniati-Adjiman:2014a,Kleniati-Adjiman:2014b,Kleniati-Adjiman:2015} can in principle handle such problems, but their broad applicability comes at a significant computational cost, limiting their practical reach to instances with very few variables and constraints. At the other end of the spectrum, recent work on integer nonlinear BOPs~\citep{Gaar-et-al:2024} has demonstrated that exploiting DCs can lead to algorithms that scale to much larger instances, but this line of work assumes the objective function of the follower problem is convex quadratic. A very recent branch-and-cut approach by~\cite{Horlaender-et-al:2025} also uses DCs to solve a larger family of integer nonlinear BOPs in which the leader and follower problems can contain nonlinear constraints as well. 
Contrary to traditional disjunctive cut approaches where DCs are derived from the convex hull of a disjunctive
set, \citet{Horlaender-et-al:2025} solve a series of subproblems where each of them is optimized
over one disjunct intersected with the integer lattice to obtain DCs.
While this framework is capable of handling a more general problem class, its generality comes at a significant computational cost, as it is not tailored to exploit the structure present in our setting, namely the polyhedral constraints at the follower and a quadratic objective function. 

Regarding non-convex follower problems without integrality requirements on the variables, there is more existing literature. For example, \citet{Bomze-et-al:2025} recently consider BOPs with a non-convex quadratic follower problem, proving $\Sigma^p_2$-hardness of this class and proposing an iterative lower- and upper-bounding scheme that is shown to be finitely convergent. 
We refer the interested reader to \cite{Bomze-et-al:2025} and the references therein for more information regarding such problems without integrality requirements on the variables.

To the best of our knowledge, no algorithm specifically designed for ILP-QILP BOPs with a non-convex follower objective function exists in the literature.
We address this gap by developing a novel branch-and-cut (\BC) solution framework based on DCs. We exploit  the algebraic structure of quadratic follower problems to derive polyhedral sets that contain no bilevel-feasible solutions in their interior. We denote such sets as bilevel-free. Contrary to the previous study by~\cite{Gaar-et-al:2024} where the follower objective function was a convex quadratic function, and the DCs were derived using a second-order cone programming (SOCP) based cut generating procedure, our new results show that DCs can be generated using LPs, even when the quadratic objective function of the follower is non-convex. 

\subsection{Contribution and Outline}

The main contribution of this paper is the development of an exact solution algorithm for ILP-QILP BOPs, namely a \BC algorithm based on DCs. We also present enhancements of this algorithm and a computational study, which shows that our algorithm outperforms state-of-the-art algorithms for problems with convex objective function and is also competitive for problems with non-convex objective (in the latter case, we compare with a solver which can tackle ILPs in the follower after using standard linearization techniques for the non-convex objective function).

The remainder of the paper is organized as follows: In Section~\ref{sec:cuts} we describe our bilevel-free set which is based on improving directions, and how to derive DCs based on it using a cut generating LP. We also describe how these cuts can be used to obtain a \BC algorithm to solve~\eqref{bilevelProblem}.
In Section~\ref{sec:separation} we are concerned with obtaining improving directions which we need to obtain cuts. We also discuss techniques to detect and remove empty disjuncts to enlarge our bilevel-free set. In Section~\ref{sec:implementation} we describe implementation details. Our computational study is presented in Section~\ref{sec:results} and Section~\ref{sec:conclusion} concludes the paper.

%%%%%%%%%%%%%

\section{Disjunctive cuts for improving-direction bilevel-free sets}\label{sec:cuts}

In this section we first present basic notation in Section~\ref{sec:notation} and \emph{improving-solution bilevel-free sets} used in earlier work in Section~\ref{sec:GaarRecap}, before introducing \emph{improving-direction bilevel-free sets} for our setting in Section~\ref{sec:bilevelFree}. We then show how to use disjunctive arguments to derive \emph{improving-direction DCs} and present the cut generating LP in Section~\ref{sec:CGLP}.
Finally, we present our \BC algorithm in Section~\ref{sec:BC}, where we also show that we are guaranteed to obtain a violated cut for any bilevel-infeasible solution under a mild condition that is always satisfied in our case, and where we establish finite convergence of the \BC.

 %%%%%%%%%%%%%

\subsection{Basic definitions and assumptions}
\label{sec:notation}
 
The optimistic BOP~\eqref{bilevelProblem}
has the so-called \emph{value function reformulation} (\VFR)
\begin{subequations}\label{vfr}
	\begin{alignat}{3}
	\min_{(x,y) \in \mathbb Z^n}  ~c^\top x + d^\top y \\%~\ell(x,y) \\
	\st~ Mx+Ny &\geq h \label{ineq:linear}\\
	  Ax+By &\geq f \label{ineq:linking}\\
     \ell \leq y &\leq u \label{ineq:ybounds}\\
	q(y) &\leq \Phi(x), \label{eq:inequValueFunction}
%	&(x,y) \in \mathbb Z^n,	\label{ineq:xyint}
	\end{alignat}
\end{subequations}
where the \emph{value function} $\Phi(x)$ of the follower problem
\begin{equation*}
\Phi(x) = {\rm min}\left\{
q(y): A x+ B y \geq f,~ \ell \le y \le u,~ y \in \mathbb{Z}^{n_2} 
%Ax+By \geq f,~  y \in \mathcal{Y}, 
%~ y \in \mathbb Z^{n_2}
%~y_j \in \mathbb Z \quad \forall j \in J
\right\} %\label{eq:follower}
\end{equation*}
is typically non-convex and non-continuous.   

Starting from the \VFR~\eqref{vfr}, we obtain the \emph{single-level relaxation} (\SLR) by 
dropping~\eqref{eq:inequValueFunction}, i.e., the optimality condition on $y$ 
for the follower problem. 
We denote the LP relaxation (i.e., replacing the integer constraint $(x,y) \in \Z^n$ with the corresponding variable bound constraints)
of the \SLR as \ccrSLR.
Given $x \in \Z^{n_1}$, let 
$$Y(x)=\{y \in \Z^{n_2}: B y \ge f-Ax,~\ell \le y \le u\}$$
be the feasible region of the follower problem for this leader decision~$x$. Moreover, let $$\mathcal F=\{(x,y) \in \Z^n: Mx + Ny \ge h,~y \in Y(x),~q(y) \leq \Phi(x)\}$$ be the feasible region of the \VFR (i.e., equivalently, $\mathcal F$ is the feasible region of our original problem~\eqref{bilevelProblem}). 

A solution $(x,y) \in \R^n$ is called \emph{(bilevel-)feasible} if it is feasible for the original problem~\eqref{bilevelProblem} (thus, if it is in  $\mathcal F$), otherwise it is called \emph{infeasible}. Additionally, a solution $(x,y) \in \R^n$ is called \emph{bilevel-infeasible} if it is feasible for the \SLR, but violates~\eqref{eq:inequValueFunction}, so bilevel-infeasibility is a special case of infeasibility. In particular, if a solution is not feasible for the SLR, it is neither bilevel-feasible nor bilevel-infeasible. We note that these terms are not used uniformly in the literature and our definition of bilevel-infeasibility follows \citet{fischetti2017new}, whereas, e.g., \citet{Gaar-et-al:2024} call every solution that is not bilevel-feasible bilevel-infeasible.
A set $C\subseteq \mathbb R^n$ is called \emph{bilevel-free} if it does not contain any bilevel-feasible solution.

Similar to \citet{fischetti2016intersection,fischetti2017new,Gaar-et-al:2024}, our aim is to design a \BC which starts with solving the \SLR and where at any time a bilevel-infeasible solution is encountered, a DC is generated to cut off this solution. Note that this means we only need to be able to guarantee that we find a cut when we encounter a solution $(x^*,y^*)$ which is integer during the \BC as continuous solutions are not feasible to the \SLR per definition (and thus cannot be bilevel-infeasible). Standard branching within the \BC deals with the integrality constraints. Nevertheless, our DCs may also be separated when $(x^*,y^*)$ is fractional, in this case the separation procedure is a heuristic. 

We make the following standard assumption for the rest of the paper (see e.g., \citet{Gaar-et-al:2024} for a similar assumption).
\begin{Assumption}\label{as:bounded}
The feasible region of \ccrSLR is a bounded polyhedron.
\end{Assumption}

%%%%%%%%%%%%%

\subsection{Improving-solution bilevel-free sets and disjunctive cuts}
\label{sec:GaarRecap}

Before introducing our new bilevel-free sets which will be the starting point to deriving our new DCs, we first describe the bilevel-free sets and DCs used in our earlier work~\citep{Gaar-et-al:2024}. The structure of the problem considered in that work is slightly different: while its leader problem is more general with also including conic constraints, its follower objective function is substantially more limited by allowing only convex quadratic objective functions by restricting $R$ to be positive semidefinite. Furthermore, while we explicitly have bound constraints $\ell \le y \le u$ in the follower problem, they only have the constraint that $y$ must be in a polyhedron.

In particular, the approach of~\citet{Gaar-et-al:2024} is as follows. Given a bilevel-infeasible solution $(x^*,y^*)$ to be cut off, the construction of the proposed bilevel-free set starts from an \emph{improving solution} $\hat y \in Y(x^*)$, i.e., a follower solution that is feasible for the follower problem at $x^*$ and strictly better than $y^*$ (typically, an optimal follower solution for the follower problem at $x^*$). The corresponding \emph{improving-solution bilevel-free set}
 is defined as
\[
S(\hat y) = \{(x,y) \in \R^n: Ax + B\hat y \geq f,~  q(y) > q(\hat y)\}.
\]
We note that any $(x,y) \in S(\hat y)$ cannot be bilevel-feasible, since $\hat y$ is a better follower response for $x$ than $y$ is. 
As for the derivation of DCs one needs sets with bilevel-free interior, the enlarged set  
\[
S^+(\hat y) = \{(x,y) \in \R^n: Ax + B\hat y \geq f - \mathbf 1,~ q(y) \geq q(\hat y)\}
\]
is used. The construction of this set exploits that all inputs and variables involved are integer.

Applying disjunctive arguments to $S^+(\hat y)$, the $m_2+1$ disjuncts
\begin{alignat*}{3}
\bar{ \mathcal{D}}_0(\hat y) : &&\qquad q(y) &\leq q(\hat y)  \\
\bar{ \mathcal{D}}_i(\hat y) : &&\qquad A_i x &\leq f_i - B_i \hat y - 1 \qquad && \forall i = 1, \dots, m_2
\end{alignat*}
are obtained, so one disjunct for the objective comparison and one disjunct for each of the $m_2$ lower-level constraints. We note that in fact,  the latter disjuncts $\bar{\mathcal{D}}_i(\hat y)$ will only appear for linking constraints, and not for the remaining lower-level constraints with $A_i = \mathbf{0}$ (and also not for the bound constraints), as the disjuncts corresponding to the remaining lower-level constraints are just trivially infeasible because $\hat y$ is a feasible follower solution. Since $\hat y$ is fixed, $\bar{\mathcal{D}}_0(\hat y)$ compares $q(y)$ to the constant $q(\hat y)$ and is therefore a quadratic disjunct. This means it is convex only when $R$ is positive semidefinite. Moreover, the separation problem in \citet{Gaar-et-al:2024} thus is an SOCP. 

Note that this construction breaks down once $R$ is not positive semidefinite, as $\bar{\mathcal{D}}_0(\hat y)$ then becomes a non-convex quadratic disjunct that is no longer amenable to an SOCP-based cut generating procedure. Closing precisely this gap is the starting point for the improving-direction bilevel-free sets we introduce next.

\subsection{Improving-direction bilevel-free sets}
\label{sec:bilevelFree}
 In this section, we generalize the characterization of  \emph{improving-direction bilevel-free sets}, which was originally introduced by \cite{watermelon} for ILP-ILP BOPs. We start by repeating their result.
\begin{Proposition}[\cite{watermelon,fischetti2017new}]
  \label{prop:wxBODY_old}
  Consider problem~\eqref{bilevelProblem} with $R=\mathbf{0}$. Then for any $\Delta\yy\in \Z^{n_2}$ such that $g^\top \Delta\yy < 0$
      holds, the
      polyhedron
  \begin{equation*}
     \bar T(\Delta\yy) = \{ (x,y) \in \R^n : Ax + B(y + \Delta\yy) \ge f,~
    \ell \le y + \Delta\yy \le u \}
  \end{equation*}
  does not contain any bilevel-feasible solution.  \end{Proposition}
Any vector $\Delta\hat{y} \in \mathbb{Z}^{n_2}$ with $g^\top \Delta\yy < 0$ (as it occurs in Proposition~\ref{prop:wxBODY_old}) is called an \emph{improving direction}: for any feasible follower solution $y \in Y(x)$, moving into direction $\Delta\hat{y}$ strictly improves the follower objective function, since $g^\top(y + \Delta\hat{y}) < g^\top y$. Moreover, due to the definition of $\bar T(\Delta\yy)$ we have that for any $(x,y) \in \bar T(\Delta\yy)$ it holds that $y + \Delta\hat{y} \in Y(x)$. This certifies that $y + \Delta\hat{y}$ is a better feasible follower solution than $y$, thus $y$ is not in $\Omega(x)$, and hence $(x, y)$ is not bilevel-feasible.

In the following, we expand  Proposition~\ref{prop:wxBODY_old} to the case when $R \neq \mathbf{0}$. Note that compared to Proposition~\ref{prop:wxBODY_old}, the condition on improving the objective function value actually becomes part of the definition of the considered set. We discuss this further after our theorem. 

\begin{Theorem}\label{le:wxBODY}
 Consider problem~\eqref{bilevelProblem}. Then for any $\Delta \hat y \in \Z^{n_2}$, the set
\[
T(\Delta \hat y) = \Defset{ (x,y) \in \R^n }{ Ax  + B(y  + \Delta \hat y) \ge f,~   \ell \le  y  +   \Delta \hat y  \le u,~ q(y+\Delta \hat{y})< q(y)
}
\]
 does not contain any bilevel-feasible solution.
\end{Theorem}
\begin{proof}
Due to $q(y+\Delta \hat{y})< q(y)$ we have that by moving from $y$ to $y + \Delta\hat{y}$ the follower objective function value is improved. Moreover, due to $Ax  + B(y  + \Delta \hat y) \ge f$ and $\ell \le  y  +   \Delta \hat y  \le u$ we have that the point $(x,y+\Delta \yy)$ is feasible for the follower problem. Hence, any $(x,y)$ in $T(\Delta \hat y)$ cannot be bilevel-feasible, as $(x,y+\Delta \yy)$ is a better solution for the follower problem compared to $(x,y)$.
\end{proof}

For $R = \mathbf{0}$, the requirement $q(y + \Delta \yy) < q(y)$ within the definition of $T(\Delta \hat y)$ in Theorem~\ref{le:wxBODY} reduces to $g^\top \Delta\hat{y} < 0$, which is exactly the condition on $\Delta\hat{y}$ in Proposition~\ref{prop:wxBODY_old}. Thus, for any $\Delta \yy$ not fulfilling this condition, the set $T(\Delta \yy)$ is empty.
Moreover, note that the set $T(\Delta \yy)$ (similar to  $\bar T(\Delta\yy)$) can contain points $(x,y)$ which are not feasible for \ccrSLR (and therefore not feasible for \SLR).

In order to be able to use DCs, we need a polyhedron whose interior contains no bilevel-feasible solution. Towards this end, we introduce an extension of $T(\Delta\hat{y})$ that yields a polyhedron that satisfies this required property.

\begin{Corollary} \label{cor:Xdelta}
 Consider problem~\eqref{bilevelProblem}. Then for any $\Delta \hat y \in \Z^{n_2}$ with $R\Delta   \hat y \neq \mathbf{0}$ or $g^\top \Delta\yy < 0$, the extended set $T^+(\Delta \hat y)$ defined as 
\begin{align*}
T^+(\Delta \hat y) = \Defset{ (x,y) \in \R^n}{ Ax  + B (y  + \Delta \hat y) \ge f-\mathbf{1},~   \ell-\mathbf{1} \le  y  +   \Delta \hat y  \le u+ \mathbf{1},~ q(y+\Delta \hat{y})\leq q(y) }
\end{align*}
is a polyhedron, even for a non-convex quadratic function $q$, and does not contain any bilevel-feasible solution in its interior $\intset(T^+(\Delta \hat y))$, for which $\intset(T^+(\Delta \hat y))\cap \Z^{n}=T(\Delta \hat y)\cap \Z^{n}$ holds. 
\end{Corollary}
\begin{proof}
 We first observe that the constraint $q(y+\Delta \hat{y}) \leq q(y)$ can be reformulated as
 \begin{align}\label{eq:reformulation}
 q(y+\Delta \hat{y}) - q(y) = 
 & (y +\Delta \hat y)^\top R(y +\Delta \hat y) + g^\top (y +\Delta \hat y) - y^\top Ry - g^\top y = \\
 & \quad 2{\Delta   \hat y}^\top R y + \Delta   \hat y^\top R\Delta   \hat y + g^\top \Delta \hat    y \le 0, \nonumber
 \end{align}  
 which is linear in $y$, thus, the set $T^+(\Delta \hat y)$ is indeed a polyhedron. Furthermore, if $ R\Delta  \hat y \neq \mathbf{0}$ and thus also ${\Delta   \hat y}^\top R \neq 0$ holds, it follows that this constraint does not vanish in the definition of $T^+(\Delta\yy)$.
 On the other hand, if $R\Delta   \hat y = \mathbf{0}$ and $g^\top \Delta\yy < 0$ holds, then this constraint is trivially strictly satisfied.

 Thus, to be in the interior of $T^+(\Delta\yy)$, for any integer solution $(x,y)$  is equivalent to satisfying
  \begin{equation*}
    Ax + B (y + \Delta\yy) > f - \mathbf{1}, 
    \qquad 
    \ell - \mathbf{1} < y+\Delta\yy < u + \mathbf{1}, 
    ~\text{and}\qquad 
    q(y+\Delta \hat{y})< q(y).
  \end{equation*}
 Due to $A \in \Z^{m_2 \times n_1}$, $B \in \Z^{m_2 \times n_2}$,  $f \in \Z^{m_2 }$, $\ell,u\in\Z^{n_2}$ and $x\in \mathbb Z^{n_1}$,$y\in \mathbb Z^{n_2}$ the
  first two conditions are equivalent to $Ax +B (y + \Delta\yy) \ge f$
  and $\ell \le y+\Delta\yy \le u$.
  As a consequence, indeed $\intset(T^+(\Delta \hat y))\cap \Z^{n} = T(\Delta \hat y)\cap \Z^{n}$ holds.
  Thus, the claim follows from Theorem~\ref{le:wxBODY}, as any bilevel-feasible solution must be integer. 
\end{proof}

The next result shows that the family of bilevel-free polyhedra $T^+(\Delta \hat y)$ is rich enough to cover every bilevel-infeasible solution in one of its interiors, which is what makes it possible to use them to design a \BC algorithm, as we do later in this section.

\begin{Theorem}\label{thm:existence}
 Consider problem~\eqref{bilevelProblem} and a bilevel-infeasible solution $(x^*,y^*)$ of it. Then, there exists a $\Delta\yy  \in \Z^{n_2}$ with $ R\Delta   \hat y \neq \mathbf{0}$ or $g^\top \Delta\yy < 0$  such that $(x^*,y^*) \in \intset(T^+(\Delta \hat y))$.
\end{Theorem}
\begin{proof}
Let $\hat y \in \arg \min \left\{
q(y) :    A x^*+ B y \geq f,~ \ell \le y \le u,~ y \in \mathbb{Z}^{n_2} 
%Ax+By \geq f,~  y \in \mathcal{Y}, 
%~ y \in \mathbb Z^{n_2}
%~y_j \in \mathbb Z \quad \forall j \in J
\right\}$ be an optimal solution to the follower problem for $x^*$.
Define $\Delta \yy=\hat y - y^*$.
Since $(x^*,y^*)$ is bilevel-infeasible, we have that $q(y^*+\Delta \hat{y})< q(y^*)$, which implies that $R\Delta   \hat y \neq \mathbf{0}$ or $g^\top \Delta\yy < 0$ with the computations done in~\eqref{eq:reformulation}. Moreover, we have that $Ax^*  + B(y^*  + \Delta \hat y) \ge f$ and $\ell \le  y^*  +   \Delta \hat y  \le u$ by definition of $\hat y$ and $\Delta \hat y$. Thus, $(x^*,y^*) \in T(\Delta \hat y)$ and consequently $(x^*,y^*) \in \intset (T^+(\Delta \hat y))$ by  Corollary~\ref{cor:Xdelta}, as any bilevel-infeasible solution is integer.
\end{proof}

Theorem~\ref{thm:existence} guarantees that a suitable improving direction $\Delta \hat y$ always exists for a given bilevel-infeasible solution. Two things remain to be established before this can be turned into an exact algorithm: that such a direction can actually be turned into a violated cut by a cut generating LP, and that by using this fact, we can obtain a B\&C algorithm. We address these points in Sections~\ref{sec:CGLP} and~\ref{sec:BC}.

\subsection{Improving-direction disjunctive cuts}
\label{sec:CGLP}

Using disjunctive arguments, Corollary~\ref{cor:Xdelta} states that for a given $\Delta \yy \in \Z^{n_2}$ with with $R\Delta   \hat y \neq \mathbf{0}$ or $g^\top \Delta\yy < 0$, there are no bilevel-feasible solutions in $\intset(T^+(\Delta \hat y))$, and thus implies that all bilevel-feasible solutions of problem~\eqref{bilevelProblem} must be in the half-spaces determined by the (linear) disjuncts 
\begin{alignat*}{3}
    \mathcal{D}_0(\Delta \hat y) : &&2{\Delta \hat y}^\top R y &\geq - \Delta \hat y^\top R\Delta \hat y - g^\top \Delta  \hat y \qquad \\
\mathcal{D}_i(\Delta \hat y) : && \qquad -A_i x - B_i y  &\geq -f_i +  B_i \Delta \hat y + 1  \quad &&\forall i=1, \dots, m_2 \\%{\color{red}\sigma \leq 0},
\mathcal{D}^\ell_j(\Delta \hat y) :  &&-y_j &\geq - \ell_j + \Delta \hat y_j + 1 &&\forall j = 1, \dots, n_2 \\
\mathcal{D}^u_j(\Delta \hat y) :  &&y_j &\geq u_j - \Delta \hat y_j +1 &&\forall j = 1, \dots, n_2, 
\end{alignat*}
which are obtained by reversing the linear inequalities that define $T^+(\Delta \hat y)$. Specifically, this means that
\begin{align} \label{eq:fsubsetD}
\mathcal F \subseteq \mathcal{D}_0(\Delta \hat{y}) \cup \mathcal{D} (\Delta \hat{y}) \cup \mathcal{D}^\ell (\Delta \hat{y}) \cup \mathcal{D}^u(\Delta \hat{y}),
\end{align}
where 
\[ \mathcal{D} (\Delta \hat{y})= \bigcup_{i=1}^{m_2}  \mathcal{D}_i(\Delta \hat{y}), \qquad \mathcal{D}^\ell (\Delta \hat{y})= \bigcup_{j=1}^{n_2}  \mathcal{D}^\ell_j(\Delta \hat{y}),~\text{and} \qquad 
\mathcal{D}^u(\Delta \hat{y}) = \bigcup_{j=1}^{n_2}  \mathcal{D}^u_j(\Delta \hat{y}). \]
The set $\mathcal{D} (\Delta \hat{y})$ contains all disjuncts derived from the lower-level constraints~\eqref{general}, whereas the sets $\mathcal{D}^\ell (\Delta \hat{y})$ and $\mathcal{D}^u (\Delta \hat{y})$ consist of all disjuncts derived from the lower-bound and upper-bound constraints~\eqref{ybounds}, respectively.

Since the disjuncts are all linear, a DC can be generated by solving an LP, as we will describe in more detail soon. 
However, this comes at the cost of involving all follower constraints on $y$ including the bound constraints in the cut generation process, reflected in the disjuncts $\mathcal{D}^\ell (\Delta\hat{y})$ and $\mathcal{D}^u (\Delta\hat{y})$. The reason is that our bilevel-free set $T^+(\Delta\hat{y})$ excludes bilevel-feasible solutions via a \emph{relative} criterion, i.e., via checking whether $y + \Delta\hat{y}$ improves upon~$y$, whereas the bilevel-free set $S(\hat{y})$ of~\cite{Gaar-et-al:2024} relies on an \emph{absolute} criterion, checking whether~$\hat{y}$ improves upon~$y$. As a consequence, our approach gives rise to $m_2 + 2n_2 + 1$ disjuncts, compared to only $m_2 + 1$ for the cut generation of~\cite{Gaar-et-al:2024}. However, their cut generation is SOCP-based, while we are only required to solve an LP to obtain a DC, which could result in a substantial computational advantage.

We now want to turn our attention to obtaining DCs based on our disjuncts 
$\mathcal{D}_0(\Delta \hat y)$, $\mathcal{D}_i(\Delta \hat y)$, $\mathcal{D}_j^\ell(\Delta \hat y)$ and $\mathcal{D}_j^u(\Delta \hat y)$.
To derive a cut generating LP (CGLP) for a given $\Delta \hat y$, we use the disjunctive programming framework of Balas (see, e.g, \citep{balas_book}). For the ease of presentation, we first aggregate all indices of different disjuncts into a single unified index set $\mathcal{K} = \{0\} \cup \{1, \dots, m_2\} \cup \{1, \dots, n_2\}^\ell \cup \{1, \dots, n_2\}^u$. Furthermore, we rewrite the linear conditions on the disjuncts as
\begin{alignat*}{7}
  \mathcal{D}_0(\Delta \hat y) : & && \phantom{- } 2\Delta\hat{y}^\top R y &&- (-\Delta\hat{y}^\top R \Delta\hat{y} - g^\top \Delta\hat{y})~&&\geq 0 &\qquad& \\
  \mathcal{D}_i(\Delta \hat y) :    &\qquad -A_i x &&- B_i y &&- (-f_i + B_i \Delta \hat{y} + 1)  &&\geq 0  && \forall i=1, \dots, m_2  \\
  \mathcal{D}_j^\ell(\Delta \hat y) :    & && -y_j &&- (-\ell_j + \Delta \hat{y}_j + 1) &&\geq 0  && \forall j=1, \dots, n_2  \\
  \mathcal{D}_j^u(\Delta \hat y) :   & && \phantom{- }y_j &&- (u_j - \Delta \hat{y}_j + 1) &&\geq 0 && \forall j=1, \dots, n_2.
\end{alignat*}
In this form it is easy to see that they can be
compactly expressed via the unified system 
\begin{equation*}
G x + H y - s \geq 0,
\end{equation*}
where the entries of the matrices $G$ and $H$, and the vector right-hand side~$s$ are defined appropriately for each type of disjunct as detailed in Table~\ref{tab:disjunction_mapping} and the rows of $G$, $H$ and the elements of $s$ are indexed by $k \in \mathcal{K}$.

\begin{table}[hbt]
\centering
\caption{Specific definition of the row-vectors $G_k$, $H_k$ and the scalar $s_k$ for each disjunct $k \in \mathcal{K}$}
\label{tab:disjunction_mapping}
\begin{tabular}{ll|ccc}
\hline
Disjunct & Disjunct type & $G_k$ & $H_k$ & $s_k$ \\ \hline
$\mathcal{D}_0(\Delta \hat y)$ & Objective  & $\mathbf{0}$ & $2\Delta\hat{y}^\top R$ & $-\Delta\hat{y}^\top R \Delta\hat{y} - g^\top \Delta\hat{y} $ \\ %\hline
$\mathcal{D}_i(\Delta \hat y)$ & Lower-level & $-A_i$ & $-B_i$ & $-f_i + B_i \Delta \hat{y} + 1$ \\ %\hline
$\mathcal{D}_j^\ell(\Delta \hat y)$ & Lower-bound  & $\mathbf{0}$ & $-e_j^\top$ & $-\ell_j + \Delta \hat{y}_j + 1$ \\ %\hline
$\mathcal{D}_j^u(\Delta \hat y)$ & Upper-bound  & $\mathbf{0}$ & $e_j^\top$ & $u_j - \Delta \hat{y}_j + 1$ \\ \hline
\end{tabular}
\end{table}

To verify whether a solution \((x^*, y^*)\) belongs to the convex hull of the disjuncts for a given $\Delta \hat y$, one may solve an LP, which involves the standard convex hull representation of a union of polyhedra. In particular, each disjunct $k \in \mathcal{K}$ is modeled using disjunct-specific auxiliary copy  variables $x^k$ and $y^k$ and a continuous weighting variable $\lambda_k \geq 0$ within the LP
%{\allowdisplaybreaks
\begin{subequations}
\label{primalcglp}
\begin{alignat}{7}
\min  ~ 0 \qquad  \\
\st  \sum_{k \in \mathcal{K}} x^k &= x^* && && (\alpha) \label{cglp:alpha}\\
 \sum_{k \in \mathcal{K}} y^k &= y^* && && (\beta) \label{cglp:beta}\\
 \sum_{k \in \mathcal{K}} \lambda_k &= 1 && && (\gamma) \label{cglp:gamma}\\
M x^k + N y^k - h \lambda_k &\geq 0 &\qquad &\forall k \in \mathcal{K} &\qquad \qquad &(\rho^k) \label{cglp:rho}\\
A x^k + B y^k - f \lambda_k &\geq 0 &&\forall k \in \mathcal{K} && (\mu^k) \label{cglp:mu}\\
y^k - \ell \lambda_k &\geq 0 &&\forall k \in \mathcal{K} && (\eta^k) \label{cglp:eta}\\
-y^k + u \lambda_k &\geq 0 &&\forall k \in \mathcal{K} && (\zeta^k) \label{cglp:zeta}\\
 G_k x^k + H_k y^k - s_k \lambda_k &\geq 0  &&\forall k \in \mathcal{K} && (\theta_k) \label{cglp:theta} \\
 \lambda_k &\geq 0 &&\forall k \in \mathcal{K}, \label{cglp:nonneg}
\end{alignat}
\end{subequations}
%}
where~\eqref{cglp:alpha}-\eqref{cglp:gamma} and~\eqref{cglp:nonneg} represent the fact that $(x^*,y^*)$ should be represented as  convex combinations of the disjunct-specific auxiliary copy variables, \eqref{cglp:rho}-\eqref{cglp:zeta} make sure that the copy variables are feasible solutions of the \ccrSLR, and finally~\eqref{cglp:theta} ensures that the copy variables are in the right disjunct.

By taking the dual of the LP~\eqref{primalcglp}, we obtain the CGLP
\begin{equation*}
\begin{aligned}
\max ~\alpha^\top x^* + \beta^\top y^* + \gamma \qquad \qquad \\
\st \alpha + M^\top \rho^k + A^\top \mu^k + G_k ^\top \theta_k &= 0 &\qquad& \forall k \in \mathcal{K} \\
 \beta + N^\top \rho^k + B^\top \mu^k + \eta^k - \zeta^k + H_k ^\top \theta_k &= 0 && \forall k \in \mathcal{K} \\
 \gamma - h^\top \rho^k - f^\top \mu^k - \ell^\top \eta^k + u^\top \zeta^k - s_k \theta_k &\leq 0 && \forall k \in \mathcal{K} \\
 \rho^k, \mu^k, \eta^k, \zeta^k, \theta_k &\geq 0 && \forall k \in \mathcal{K},
\end{aligned}
\tag{CGLP}
\label{eq:CGLP}
\end{equation*}
which directly generates valid  cut coefficients $\alpha, \beta$ and $\gamma$.

If the LP~\eqref{primalcglp} is feasible, then $(x^*,y^*)$ belongs to the convex hull of the union of the disjuncts $\mathcal{D}_k(\Delta \hat{y})$ for $k \in \mathcal{K}$, each intersected with the feasible region of \ccrSLR, and hence no inequality that is valid for this convex hull cuts off $(x^*,y^*)$. In contrast, if~\eqref{primalcglp} is infeasible, then the constraints of~\eqref{eq:CGLP} define a homogeneous cone, which makes the objective function unbounded unless a normalization condition is introduced. To ensure a bounded, non-trivial optimal solution, we append the standard normalization condition

\begin{equation*}
    \sum_{k \in \mathcal{K}} \left( \mathbf{1}^\top \rho^k + \mathbf{1}^\top \mu^k + \mathbf{1}^\top \eta^k + \mathbf{1}^\top \zeta^k + \theta_k \right) \leq 1
\end{equation*}
to~\eqref{eq:CGLP} which acts as an $L_1$-norm constraint on the non-negative dual multipliers across all disjuncts (see, e.g., \citet{Gaar-et-al:2024}).
Then, by solving~\eqref{eq:CGLP} we obtain the cut coefficients $\alpha, \beta$ and $\gamma$ such that $\alpha^\top x + \beta^\top y + \gamma \leq 0$  is a DC that tries to separate $(x^*,y^*)$ from the convex hull of $\mathcal{D}_0(\Delta \hat{y}) \cup \mathcal{D} (\Delta \hat{y}) \cup \mathcal{D}^\ell (\Delta \hat{y}) \cup \mathcal{D}^u(\Delta \hat{y})$.

%%%%%%%%%%%%%

\subsection{A branch-and-cut algorithm}
\label{sec:BC}

We now turn our attention to describing how to use~\eqref{eq:CGLP} within a \BC algorithm. To do so, we still need to establish that in any \BC node, for an improving direction $\Delta \hat y$ (obtained for example like it is done in the proof of Theorem~\ref{thm:existence}), \eqref{eq:CGLP} is actually guaranteed to return a DC violated by the bilevel-infeasible solution $(x^*,y^*)$ under consideration.
In order to achieve that, let
\begin{align}\label{previousCuts}
    P x + Q y \ge r
\end{align} 
for appropriately set $P$, $Q$ and $r$, 
contain any \DCs derived at any ancestor \BC node and the local variable bounds imposed by previous branching decisions at the current node of the \BC.
Then let
\begin{equation*}
\mathcal P = \Defset{(x,y) \in \R^n}{\eqref{ineq:linear},\eqref{ineq:linking},\eqref{ineq:ybounds}, \eqref{previousCuts}},
\end{equation*}
so
$\mathcal P \subseteq \R^n$ denotes the current node relaxation of the \SLR of~\eqref{bilevelProblem} 
within our B\&C algorithm for solving~\eqref{bilevelProblem}. 
Note that by construction, $\mathcal P$ is a subset of \ccrSLR.

Theorem~\ref{thm:existence} guarantees that, for a bilevel-infeasible solution $(x^*,y^*)$, some $\Delta \hat y$ exists such that $(x^*,y^*)$ lies in the interior of $T^+(\Delta \hat y)$.
Furthermore, by Corollary~\ref{cor:Xdelta} the interior of $T^+(\Delta \hat y)$ does not contain any bilevel-feasible solution, while we constructed $\mathcal D_0(\Delta \hat y) \cup \mathcal D(\Delta \hat y) \cup \mathcal D^\ell(\Delta \hat y) \cup \mathcal D^u(\Delta \hat y)$ in such a way that all bilevel-feasible solutions are contained.
These, however, are statements about the sets themselves and do not yet guarantee that~\eqref{eq:CGLP} (that tries to separate $(x^*,y^*)$ from the convex hull of the disjunction) can actually produce a cut violated by $(x^*,y^*)$. The following theorem closes this gap, provided $(x^*,y^*)$ is an extreme point of $\mathcal P$. This is a mild requirement for usage within a \BC, as it is automatically satisfied whenever $(x^*,y^*)$ is returned as the optimal solution of a node relaxation solved with a simplex-type LP solver.

\begin{Theorem}\label{thm:separation}
Let $(x^*,y^*)$ be a bilevel-infeasible solution that is an extreme point of $\mathcal P$. Let $\Delta \hat y \in \Z^{n_2}$ with $R\Delta   \hat y \neq \mathbf{0}$ or $g^\top \Delta\yy < 0$ be such that $(x^*,y^*) \in \intset (T^+(\Delta \hat y))$. Then, there is a valid inequality for $\mathrm{conv}\big(\bigcup_{k \in \mathcal K} (\mathcal P \cap \mathcal D_k(\Delta \hat y))\big)$ that is violated by $(x^*,y^*)$, and it can be obtained by solving \eqref{eq:CGLP}.
\end{Theorem}

\begin{proof}
First, note that a $\Delta \hat y$ with  $(x^*,y^*) \in \intset (T^+(\Delta \hat y))$ as claimed exists by Theorem~\ref{thm:existence}. Then,
since $(x^*,y^*) \in \intset (T^+(\Delta \hat y))$, the solution $(x^*,y^*)$ strictly satisfies each of the defining inequalities of $T^+(\Delta \hat y)$.  Equivalently, $(x^*,y^*)$ does not satisfy any inequality that defines a disjunct $\mathcal D_k(\Delta \hat y)$.

Assume, for contradiction, that no inequality separating $(x^*,y^*)$ from $\mathrm{conv}\big(\bigcup_{k \in \mathcal K} (\mathcal P \cap \mathcal D_k(\Delta \hat y))\big)$ exists. Then $(x^*,y^*)$ lies in this convex hull, i.e.,
\[
(x^*,y^*) = \sum_{k \in \mathcal K} \lambda_k (x^k,y^k),
\]
for some $\lambda_k \ge 0$ with $\sum_{k \in \mathcal K} \lambda_k = 1$ and $(x^k,y^k) \in \mathcal P \cap \mathcal D_k(\Delta \hat y)$ whenever $\lambda_k > 0$. Since $(x^*,y^*)$ satisfies no single disjunct, we have $(x^k,y^k) \neq (x^*,y^*)$ for every $k$ with $\lambda_k>0$, and thus at least two of the $\lambda_k$ must be strictly positive (otherwise $(x^*,y^*) = (x^k,y^k)$ for the unique $k$ with $\lambda_k=1$, a contradiction). Hence $(x^*,y^*)$ is a proper convex combination of two or more points of $\mathcal P$, contradicting that $(x^*,y^*)$ is an extreme point of $\mathcal P$.

Thus, a separating inequality exists and the LP~\eqref{primalcglp} is infeasible for the given $(x^*,y^*)$. By LP duality, its dual \eqref{eq:CGLP} is then unbounded once the normalization constraint is dropped, and with the normalization constraint added it attains a finite optimal solution with a strictly positive objective function value, which yields the coefficients of the desired \DC.
\end{proof}

%%%%%%%%%%%%%

Theorems~\ref{thm:existence} and~\ref{thm:separation} together justify the usage of a \BC algorithm for solving problem~\eqref{bilevelProblem}. We initialize our \BC with the \ccrSLR as the relaxation $\mathcal P$ at the root node. Whenever an integer-feasible extreme point $(x^*,y^*)$ of the current node relaxation $\mathcal P$ is encountered, we need to (i) check if the point is bilevel-feasible (and thus a candidate to update the incumbent), and if not, (ii) search for a direction $\Delta \hat y \in \Z^{n_2}$ with $(x^*,y^*) \in \intset (T^+(\Delta \hat y))$ (whose existence is guaranteed by Theorem~\ref{thm:existence}) and finally (iii) obtain a violated \DC via solving the separation problem \eqref{eq:CGLP} for this $\Delta \hat y$.  This \DC is then added to $\mathcal P$, remains valid throughout the \BC subtree rooted at the current node, and $\mathcal{P}$ is resolved. A node is fathomed once its relaxation $\mathcal P$ is infeasible or its bound does not improve upon the incumbent objective function value.

In our implementation, steps (i) and (ii) are combined into the \emph{improving-direction generation} (IDG), which involves checking bilevel-feasibility and trying to obtain an improving direction $\Delta \hat y$ for a given solution $(x^*,y^*)$. If $(x^*,y^*)$ is bilevel-infeasible, by Theorem~\ref{thm:existence} the IDG is guaranteed to yield such an improving direction, thus if the IDG does not find an improving direction, this is a certificate that the solution $(x^*,y^*)$ is bilevel-feasible. The combination of the IDG and the separation using \eqref{eq:CGLP} in case that the solution is bilevel-infeasible, i.e., the entirety of steps (i), (ii) and (iii), is denoted as \emph{separation procedure}. We note that a simple way to perform the IDG consists of just solving the follower problem (see the proof of Theorem~\ref{thm:existence}), however, we also propose other ways the IDG can be performed, see Section~\ref{sec:separation} for details.

Next, we show that our proposed \BC works as intended.

\begin{Theorem}\label{thm:convergence}
Under Assumption~\ref{as:bounded}, the \BC described above solves problem~\eqref{bilevelProblem} to optimality in a finite number of iterations.
\end{Theorem}
\begin{proof}
Suppose the \BC terminates but returns a solution $(x^*,y^*)$ that is bilevel-infeasible. This is not possible, because by Theorem~\ref{thm:existence}, an improving direction $\Delta \hat y$ with $(x^*,y^*) \in \intset (T^+(\Delta \hat y))$ exists, and by Theorem~\ref{thm:separation} \eqref{eq:CGLP} then produces a \DC violated by $(x^*,y^*)$, so $(x^*,y^*)$ would have been cut off rather than returned as the final optimal solution.

Suppose the \BC terminates with a bilevel-feasible solution that is not optimal. This is not possible either, as by Corollary~\ref{cor:Xdelta} and the definition of our disjunctions, every \DC added throughout the \BC preserves all bilevel-feasible solutions because of~\eqref{eq:fsubsetD}.
Hence no \DC ever removes a bilevel-feasible solution from the relaxation of any node, and in particular an optimal solution of problem~\eqref{bilevelProblem} remains feasible for the relaxation of some active (i.e., non-fathomed) node until it is found or an equally good incumbent has already been recorded.

Finally, suppose the \BC does not terminate.  By Assumption~\ref{as:bounded}, \ccrSLR is a bounded polyhedron and thus the \SLR has a bounded feasible region, so branching on bounded integer variables produces a tree of finite depth. Whenever an integer solution which is not bilevel-feasible is encountered as solution at a  
node, Theorem~\ref{thm:separation} guarantees that the added \DC is violated by the integer solution that generated it, thereby removing that solution from this node relaxation. Since the bounded node relaxation contains only finitely many integer solutions, the processing of each node terminates after finitely many separation rounds. Finite depth of the tree together with finitely many DC separation rounds per node in the tree contradicts non-termination.
\end{proof}

In our \BC described above fractional extreme points $(x^*,y^*)$ of the current node relaxation $\mathcal P$ are not separated, as they can be handled by standard branching. However, in our implementation (see Section~\ref{sec:ImplementationDIG} for details), we also separate fractional extreme points to speed up the computation and strengthen the node relaxation $\mathcal{P}$. For such extreme points our separation procedure is a heuristic in the sense that we are not guaranteed to find a violated DC.

%%%%%%%%%%%%%

\section{Separation of improving-direction disjunctive cuts}\label{sec:separation}
In this section we discuss how to perform the IDG of our separation procedure, i.e., how to search for an improving direction $\Delta \hat y$ for a given solution $(x^*,y^*)$ or confirm its bilevel-feasibility. Before detailing the IDG, however, we note that it is often possible to reduce the size of the CGLP by dropping disjuncts that are not needed for the correctness of our approach. We first discuss such disjuncts in Section~\ref{section:redundantDisjunctions}, since two of the procedures we later use for IDG are built directly around identifying such  disjuncts. We then detail, in Sections~\ref{sec:sepMINLP} through~\ref{sec:sepHeuristic}, four different strategies for IDG.

%%%%%%%%%%%%%

\subsection{Redundant disjuncts}
\label{section:redundantDisjunctions}
For any $k \in \mathcal K$, the disjunct $\mathcal D_k(\Delta \hat y)$ can be dropped from $\mathcal{D}_0(\Delta \hat{y}) \cup \mathcal{D} (\Delta \hat{y}) \cup \mathcal{D}^\ell (\Delta \hat{y}) \cup \mathcal{D}^u(\Delta \hat{y})$ and thus also from the CGLP without endangering the correctness of our approach, if no bilevel-feasible solution that could still improve the objective function value upon the current incumbent lies in $\mathcal P \cap \mathcal D_k(\Delta \hat y)$, because in this way we can guarantee to not lose any relevant bilevel-feasible solution. We call such a disjunct $\mathcal D_k(\Delta \hat y)$ \emph{redundant}. Removing redundant disjuncts reduces the size of the CGLP and may even strengthen the resulting cut.
We will now discuss several options of removing redundant disjuncts from the literature. For the ease of readability, we will present these results already adapted to our setting.

First, we present a basic result of~\cite{Gaar-et-al:2024} adjusted to our specific problem setting.
\begin{Lemma}
\label{lemma:generalRedundancy}
For any $\Delta \hat y \in \Z^{n_2}$ and for each $k \in \mathcal{K}$, the disjunct $\cD_k(\Delta \hat y)$ is redundant
\begin{itemize}
    \item if $\cD_k(\Delta \hat y) \cap \cP = \emptyset$ (LP-feasibility), or    
    \item if $\cD_k(\Delta \hat y) \cap \cP \cap \mathbb{Z}^{n} = \emptyset$ (IP-feasibility), or
    \item if $\cD_k(\Delta \hat y) \cap \cP \cap \mathcal{F} = \emptyset$ 
    (bilevel-feasibility).
\end{itemize}
\end{Lemma}
Since $\mathcal{F} \subseteq \mathbb Z^n$ (every bilevel-feasible solution is integer), each condition of Lemma~\ref{lemma:generalRedundancy} trivially implies the next one: LP-feasibility is the strongest and rarest to hold, while bilevel-feasibility is the weakest and holds most often, so ideally we would always check the latter to remove as many redundant disjuncts as possible. However, checking bilevel-feasibility exactly is as difficult as solving problem~\eqref{bilevelProblem} itself, and even checking LP-feasibility exactly requires solving an LP. The next two results give closed-form sufficient conditions for LP-feasibility that do not require solving any LP, by relaxing $\mathcal P$ to a superset that is easier to deal with.
The first result is adapted from \cite{fischetti2017new} and can be applied to all of our lower-level disjuncts $\cD_i(\Delta \hat y)$ for any $i \in \{1, \dots, m_2\}$.

\begin{Lemma}
For any $\Delta \hat y \in \Z^{n_2}$ and  $i \in \{1, \dots, m_2\}$, if $B_i \Delta \hat y \ge 0$ then the disjunct 
$\cD_i(\Delta \hat y)$ is redundant.
\label{lemma:Fischetti17redundant}
\end{Lemma}
Note that for $i \in \{1, \dots, m_2\}$ Lemma~\ref{lemma:Fischetti17redundant} is the special case of LP-feasibility of Lemma~\ref{lemma:generalRedundancy} obtained by relaxing $\mathcal P$ to the single half-space $\{(x,y)  \in \R^n: A_i x + B_i y \ge f_i\}$. So, Lemma~\ref{lemma:Fischetti17redundant} is indeed a computationally cheaper alternative for checking the redundancy of lower-level disjuncts.

For the second alternative, recall that improving-direction DCs are locally valid cuts (as they depend on relaxation $\mathcal P$ at the current node of the B\&C tree), so any additional information available at the current node of the B\&C tree, in particular tighter local variable bounds, can be exploited when checking redundancy. This is exactly the idea behind the next result which is an adaptation of a result from \cite{fischetti2018use} and \cite{Gaar-et-al:2024}, to our setting.

\begin{Lemma}
\label{lemma:unifiedRedundancy}
Let $x^-,x^+ \in \Z^{n_1}$ and $y^-,y^+ \in \Z^{n_2}$ be the current lower and upper bounds on the leader and follower variables within one node of the \BC tree, respectively. 
Let $\Delta \hat y \in \Z^{n_2}$.

\begin{enumerate}
    \item \label{lemma:unifiedRedundancy:gen}
Let $k \in \mathcal{K}$ and recall from Section~\ref{sec:CGLP} that the disjunct $\mathcal D_k(\Delta \hat y)$ can be written as $G_k x + H_k y \ge s_k$ for the corresponding $G_k$, $H_k$ and $s_k$ given in Table~\ref{tab:disjunction_mapping}. 
 If
\[
\sum_{j=1}^{n_1} \max\{ G_{kj} x^-_j, G_{kj} x^+_j\} + \sum_{j=1}^{n_2} \max\{ H_{kj} y^-_j, H_{kj} y^+_j\} \le s_k -1
\]
holds, then the disjunct $\mathcal D_k(\Delta \hat y)$ is redundant.

\item \label{lemma:unifiedRedundancy:bounds}
For any $j\in \{1,\dots,n_2\} $, if $\Delta \hat y_j > \ell_j - y_j^- -1$ holds, then the lower-bound disjunct $\mathcal{D}^\ell_j(\Delta \hat y)$ is redundant and if $\Delta \hat y_j < u_j - y_j^+ +1$ holds, then the upper-bound disjunct $\mathcal{D}^u_j(\Delta \hat y)$ is redundant.

\item \label{lemma:unifiedRedundancy:linking}
For any $i \in \{1, \dots, m_2\}$, if
\[ \sum_{j=1}^{n_1} \min\{ A_{ij} x^-_{j}, A_{ij} x^+_{j}\}  + \sum_{j=1}^{n_2} \min\{ B_{ij} y^-_j, B_{ij} y^+_j\} \ge f_i - B_i \Delta \hat y \]
holds, then the lower-level disjunct $\cD_i(\Delta \hat y)$ is redundant.
\end{enumerate}
\end{Lemma}
\begin{proof}
    The first part follows from the results of \cite{Gaar-et-al:2024}. 
    
    The second part on the lower and upper-bound disjuncts for $j\in \{1,\dots,n_2\}$ is a direct consequence of the first part with $k = j$ with using the fact that $(G_k,H_k,s_k) = (\mathbf{0}, -e_j^\top, -\ell_j + \Delta \hat y_j + 1)$ for $\mathcal D_j^\ell(\Delta \hat y)$ and $(\mathbf{0}, e_j^\top,  u_j - \Delta \hat y_j + 1)$ for $\mathcal D_j^u(\Delta \hat y)$. 

    Finally, the third part on the lower-level disjuncts for $i \in \{1, \dots, m_2\}$ follows from the first part with $k=i$ with $(G_k,H_k,s_k) = (-A_i, -B_i, -f_i + B_i \Delta \hat y + 1)$.
\end{proof}

Observe that Lemma~\ref{lemma:unifiedRedundancy} is the special case of LP-feasibility in Lemma~\ref{lemma:generalRedundancy} obtained by relaxing $\mathcal P$ to the box $[x^-,x^+]\times[y^-,y^+]$.

A different kind of argument, not based on emptiness of $\cD_k(\Delta \hat y) \cap \cP \cap \mathcal{F}$, but on the current incumbent bound, gives a further, complementary source of redundant disjuncts. In order to derive the result we have again adapted a result of \cite{Gaar-et-al:2024} to our setting.

\begin{Lemma}
\label{lemma:RO}
Let $\mathrm{UB}$ be the objective-function value of the best-known feasible solution for problem~\eqref{bilevelProblem}. 
Let $\Delta \hat y \in \Z^{n_2}$ and let $k \in \mathcal{K}$ such that $\mathcal P \cap \mathcal D_{k}(\Delta \hat y) \cap \mathbb{Z}^n \neq \emptyset$.
Let $\bar U_{k}$ be the optimal objective-function value of
\[
\min\{c^\top x + d^\top y : (x,y) \in \mathcal P \cap \mathcal D_{k}(\Delta \hat y) \cap \mathbb{Z}^n\}.
\]
If $\bar U_{k} \ge \mathrm{UB}$, then the disjunct $\mathcal D_{k}(\Delta \hat y)$ is redundant.
\end{Lemma}

Observe that it follows from the IP-feasibility condition of Lemma~\ref{lemma:generalRedundancy} that also in the case $\mathcal P \cap \mathcal D_{k}(\Delta \hat y) \cap \mathbb{Z}^n = \emptyset$ the disjunct $\mathcal D_{k}(\Delta \hat y)$ is redundant.

As a final note to this section, we want to mention that Theorem~\ref{thm:convergence} certifying the correctness of our \BC is still true, even when redundant disjuncts are removed by the Lemmas~\ref{lemma:generalRedundancy}, \ref{lemma:Fischetti17redundant}, \ref{lemma:unifiedRedundancy} and~\ref{lemma:RO}. 
This is the case because even though the proof of Theorem~\ref{thm:convergence} was formulated using the fact that no bilevel-feasible solution is cut away (which could happen when removing redundant disjuncts with Lemma~\ref{lemma:RO}), the proof also goes through when using the fact that no bilevel-feasible solution that is better than the current incumbent is cut away, which holds for all of our lemmas.

%%%%%%%%%%%%%
\subsection{Improving-direction generation}
\label{sec:improvingDir}
Next, we consider four different strategies for IDG, where three of them include solving a mixed-integer non-linear program (MINLP). Two MINLP-approaches based on producing many redundant disjuncts are presented in Section~\ref{sec:sepMINLP}, and one MINLP-approach involving an improving follower solution is discussed in Section~\ref{sec:sepFeasibility}. Finally,  one heuristic procedure, that is not guaranteed to produce an improving direction for a given bilevel-infeasible solution and that reuses previously found improving directions, is presented in Section~\ref{sec:sepHeuristic}.

For ease of exposition, our description in the following focuses on the case where the solution $(x^*,y^*)$ which we want to separate is feasible for the \SLR, and therefore integer, as successful separation is required only in this case for the correctness of our \BC algorithm. We note that all IDG strategies can also be called for fractional solutions feasible for \ccrSLR as a heuristic. Even though fractional solutions are inherently infeasible due to violating integrality and will eventually be cut off via branching, they may still lie in a bilevel-free set defined by an improving direction, and thus can be cut off earlier by a \DC. 

\subsubsection{Exploiting redundancy by solving MINLPs}
\label{sec:sepMINLP}
Given a solution $(x^*,y^*)$ feasible for the \SLR, along the lines of \cite{fischetti2017new} we propose two MINLPs that decide bilevel-feasibility of $(x^*,y^*)$ and yield an improving direction $\Delta \hat y$ while aiming to render many disjuncts redundant, if $(x^*,y^*)$ is bilevel-infeasible. 
The two MINLPs for IDG both use the integer variable $\Delta y \in \Z^{n_2}$ to represent an improving direction and differ only in how they encode disjunct redundancy: the first one uses continuous slack variables, the second one uses binary indicator variables. To describe them, we associate with any direction $\Delta y \in \Z^{n_2}$ the vector
\[
v(\Delta y) = \begin{pmatrix}
\max\{0, -B_i \Delta y\}_{i=1,\dots,m_2}\\
\max\{0, -\Delta y_j\}_{j=1,\dots,n_2} \\
\max\{0, \Delta y_j\}_{j=1,\dots,n_2} \end{pmatrix} \in \R^{m_2 + 2n_2}_{\ge 0},
\]
whose entries are zero when the corresponding disjunct $\mathcal{D}_i(\Delta y)$, $\mathcal{D}_j^\ell(\Delta y)$ or $\mathcal{D}_j^u(\Delta y)$ is redundant. 

\paragraph{Slack-MINLP}
In our first approach for IDG the auxiliary slack variables $t_i$ for $i \in \{1,\dots,m_2\}$, and $r_j$, $s_j$ for $j\in\{1,\dots,n_2\}$, are continuous and equal to zero when the corresponding disjunct $\mathcal{D}_i(\Delta  y)$, $\mathcal{D}_j^\ell(\Delta  y)$ and $\mathcal{D}_j^u(\Delta  y)$ is redundant, i.e., when $B_i \Delta  y \ge 0$, $\Delta  y_j \geq 0$, and $\Delta  y_j \leq 0$ holds, respectively (Lemma~\ref{lemma:Fischetti17redundant} and Lemma~\ref{lemma:unifiedRedundancy}.\ref{lemma:unifiedRedundancy:bounds} with $y^- = \ell$ and $y^+ = u$). We then search for an improving direction $\Delta y$ with a small $L_1$-norm of $v(\Delta y)$ with the MINLP
\begin{subequations}
\label{eq:NewDeltayHat_1}
\begin{alignat}{3}
\min ~\sum_{i=1}^{m_2}  t_i + \sum_{j=1}^{n_2} (r_j+s_j) \quad \qquad \label{eq:NewDeltayHat_1_Obj} &&\\
\st \Delta y^\top  R \Delta y + 2 \Delta y^\top  R y^*  { + g^\top  \Delta y}  & \le -1  && \label{eq:NewDeltayHat_1_OFviol}\\
- B_i \Delta y &  \le t_i,  && \forall i =1, \dots, m_2   \label{eq:NewDeltayHat_1_FolCon_0} \\
 -B_i \Delta y & \le A_i x^* + B_i y^*  - f_i &\qquad &  \forall  i =1, \dots, m_2    \label{eq:NewDeltayHat_1_FolCon_Feas} \\
 -   \Delta y_j  & \le r_j  && \forall j=1, \dots, n_2 \label{eq:NewDeltayHat_1_Low_0} \\
 -   \Delta y_j  & \le  y^*_j - \ell_j && \forall  j=1, \dots, n_2 \label{eq:NewDeltayHat_1_Low_Feas} \\
 \Delta y_j  & \le s_j && \forall  j=1, \dots, n_2    \label{eq:NewDeltayHat_1_Up_0} \\
 \Delta y_j  & \le u_j - y^*_j  && \forall  j=1, \dots, n_2   \label{eq:NewDeltayHat_1_Up_Feas} \\
 \Delta y & \text{ integer},   (t,r,s) \ge 0,  \label{eq:NewDeltayHat_1_VarDomain}
\end{alignat}    
\end{subequations}
where at an optimal solution $t_i = \max\{0, -B_i \Delta y\}$, $r_j = \max\{0, -\Delta y_j\}$ and $s_j = \max\{0, \Delta y_j\}$ hold, so the objective function~\eqref{eq:NewDeltayHat_1_Obj} minimizes  $\|v(\Delta y)\|_1$, which serves as a convex surrogate for the number of non-redundant disjuncts.
Furthermore, the non-convex quadratic constraint~\eqref{eq:NewDeltayHat_1_OFviol} makes sure that $\Delta y$ is an improving direction and
the constraints~\eqref{eq:NewDeltayHat_1_FolCon_Feas}, and \eqref{eq:NewDeltayHat_1_Low_Feas} and~\eqref{eq:NewDeltayHat_1_Up_Feas} ensure that $y^* + \Delta y $ is a feasible solution of the follower problem (given that $y^*$ is integer, which is the only case where we need to make sure that an improving direction is found for the correctness of our approach). 

Moreover, the constraints~\eqref{eq:NewDeltayHat_1_FolCon_0}, \eqref{eq:NewDeltayHat_1_Low_0} and~\eqref{eq:NewDeltayHat_1_Up_0} allow the variables $t$, $r$ and $s$ to be zero if and only if the condition for redundancy of the corresponding disjunct is satisfied. 
In particular, $t_i=0$ forces $-B_i \Delta y \le 0$ via~\eqref{eq:NewDeltayHat_1_FolCon_0}, so the $i$-th disjunct $\mathcal D_i(\Delta y)$ is then redundant by Lemma~\ref{lemma:Fischetti17redundant}. Similarly, $r_j=0$ forces $\Delta y_j \ge 0$, and $s_j=0$ forces $\Delta y_j \le 0$, so the corresponding lower-bound or upper-bound disjunct for component $j$ is redundant by Lemma~\ref{lemma:unifiedRedundancy}.\ref{lemma:unifiedRedundancy:bounds}.

\paragraph{Binary-MINLP} 
In our second approach for IDG, the binary variables $\tau_i$ for $i \in \{1,\dots,m_2\}$, and $\rho_j$, $\sigma_j$ for $j\in\{1,\dots,n_2\}$ play a similar role as $t_i$, $r_j$, and $s_j$ above: a value of zero indicates that the corresponding disjunct is redundant. 
We now search for an improving direction $\Delta y$ that renders as many disjuncts as possible redundant, this time by using multiplicative constraints instead of additive slacks to enforce this, with the MINLP
\begin{subequations}
\label{eq:NewDeltayHat_2}
\begin{alignat}{3}
\min ~ \sum_{i=1}^{m_2}  \tau_i  + \sum_{j=1}^{n_2} (\rho_j+\sigma_j) \qquad\quad && \label{eq:NewDeltayHat_2_Obj}\\
\st  \Delta y^\top  R \Delta y + 2 \Delta y^\top  R y^* + g^\top  \Delta y & \le -1  && \label{eq:OFviol_1}\\
 -B_i \Delta y & \le (A_i x^* + B_i y^*  - f_i) \tau_i &\qquad&  \forall i =1, \dots, m_2    \label{eq:NewDeltayHat_2_FolCon_Feas} \\
-   \Delta y_j  & \le  (y^*_j - \ell_j)\rho_j &&  \forall j=1, \dots, n_2    \label{eq:3_1} \\
 \Delta y_j  & \le (u_j - y^*_j)\sigma_j  &&  \forall j=1, \dots, n_2   \label{eq:2_1} \\
 \Delta y & \text{ integer},   (\tau,\rho,\sigma) \text{ binary},
\end{alignat}    
\end{subequations}
where at an optimal solution $\tau_i = 0$ if and only if $B_i \Delta y \ge 0$, $\rho_j = 0$ if and only if $\Delta y_j \ge 0$, and $\sigma_j = 0$ if and only if $\Delta y_j \le 0$ hold, so the objective function~\eqref{eq:NewDeltayHat_2_Obj} minimizes the number of nonzero entries of $v(\Delta y)$. Thus, in contrast to the $L_1$-surrogate of~\eqref{eq:NewDeltayHat_1}, the MINLP~\eqref{eq:NewDeltayHat_2} exactly maximizes the number of redundant disjuncts.
The reasoning behind the redundancy checks in~\eqref{eq:NewDeltayHat_2} is analogous to the one of~\eqref{eq:NewDeltayHat_1} transformed to the multiplicative setting.
Note that the improving-direction constraint~\eqref{eq:OFviol_1} is again non-convex quadratic in $\Delta y$. Together with the remaining constraints, it ensures that $(x^*,y^*) \in T(\Delta y)$, i.e., that $y^* + \Delta y$ is a better follower response for $x^*$ than $y^*$ (if $y^*$ is integer).

By construction, a feasible solution to either MINLP formulation directly yields a $\Delta \hat y$ with $q(y^*+\Delta \hat y) < q(y^*)$ and $y^*+\Delta \hat y$ feasible for the follower problem at $x^*$, i.e., an improving direction $\Delta \hat y$ witnessing that $(x^*,y^*)$ is bilevel-infeasible. Note, that the integrality of both $R$ and $g$ is essential for taking into account the improving-direction constraints~\eqref{eq:NewDeltayHat_1_OFviol} and~\eqref{eq:OFviol_1} in a correct way. If these MINLPs are infeasible, then there is no improving direction for $(x^*,y^*)$, implying that it is a bilevel-feasible solution. Note that such an observation is also made in \citet{Battista-Ralphs:2025} for mixed-integer bilevel linear problems.

%%%%%%%%%%%%%

 \subsubsection{Using an improving follower solution obtained by solving an MINLP}
 \label{sec:sepFeasibility}
Instead of directly searching for an improving direction for IDG, one could search for an improving follower solution for a given solution $(x^*,y^*)$ feasible for the \SLR, i.e., a better feasible follower response than $y^*$ with respect to $x^*$. If such a solution~$\hat y$ exists, an improving direction is given as 
\begin{align}\label{eq:formulahaty}
    \Delta \hat y = \hat y - y^*.
\end{align} 
Hence, instead of solving one of the above MINLPs for IDG, one can search for a feasible follower solution~$\hat y$ such that
\begin{align}\label{eq:yStarBetter}
    q(y^*) > q(\hat y) 
\end{align}
holds.
This can be done by solving the follower problem, which is also an MINLP, to optimality, or until we find a feasible follower solution which satisfies~\eqref{eq:yStarBetter}, see \cite{Gaar-et-al:2024} for more details. Finally, if the optimal follower objective value $\Phi(x^*)=q(y^*)$, then $(x^*,y^*)$ is bilevel-feasible by definition, so also when using this follower-problem-based approach for IDG the bilevel-feasibility check is already included implicitly.

%%%%%%%%%%%%%

\subsubsection{Using a heuristic approach by utilizing previous improving directions}
\label{sec:sepHeuristic}
The last approach we consider for IDG for a given solution $(x^*,y^*)$ feasible for the \SLR is a heuristic method in the sense that we are not guaranteed to find an improving direction if one exists, unlike all previously described IDG approaches, where such a guarantee holds. In our heuristic method we re-use
previously generated improving directions $\Delta \hat y$ for IDG, by employing an \emph{improving-direction pool}, which we denote by $\mathcal{C}$. 
Given a non-empty pool $\mathcal{C}$, one can iterate through the directions of the pool to detect a feasible and improving one. For a direction $\Delta y' \in \mathcal{C}$, this involves

\begin{enumerate}
\item checking the feasibility of $y^*+\Delta y'$ for the follower problem at $x^*$, and 
\item verifying that $\Delta y'$ is an improving direction at $(x^*,y^*)$, i.e., ${\Delta y'} ^\top R \Delta y' + 2 \Delta y'^\top R y^* + g^\top \Delta y' \le -1$.
\end{enumerate}

Any direction that satisfies conditions 1.\ and 2.\ is a potential output of our IDG and can be employed to continue the separation procedure by solving~\eqref{eq:CGLP}. If no such direction is present, then one can proceed with one of the exact approaches described previously and then add this newly found improving direction to the pool~$\mathcal{C}$. Also in this approach any improving direction is a certificate for bilevel-infeasibility of $(x^*,y^*)$ and thus an explicit bilevel-feasibility check is not necessary.

%%%%%%%%%%%%%
\section{Algorithmic details}\label{sec:implementation}

In this section, we describe the implementation details of the \BC algorithm introduced in Section~\ref{sec:BC}, including the IDG strategy in Section~\ref{sec:ImplementationDIG}, disjunct redundancy checks in Section~\ref{sec:redundancy_check}, details on constructing and solving the CGLP in Section~\ref{sec:implCGLP} and maintaining the improving-direction pool in Section~\ref{sec:implPool}. %

\subsection{Improving-direction generation strategy}
\label{sec:ImplementationDIG}

Within the separation procedure at a \BC node, we first have to determine which improving direction~$\Delta\hat y$ to use.
We consider the following main IDG strategies based on the findings of Section~\ref{sec:improvingDir} for our B\&C algorithm: 

\begin{itemize}
    \item \texttt{S} : Solve the Slack-MINLP~\eqref{eq:NewDeltayHat_1} to directly obtain $\Delta\hat y$. 
\item \texttt{B}: Solve the Binary-MINLP~\eqref{eq:NewDeltayHat_2} to directly obtain $\Delta\hat y$.
\item  \texttt{F} : Solve the follower problem~\eqref{llProblem} to get $\hat{y}$, and then calculate $\Delta\hat y$ via~\eqref{eq:formulahaty}.
    
\end{itemize}

For each IDG strategy, the associated MINLP~\eqref{eq:NewDeltayHat_1}, \eqref{eq:NewDeltayHat_2} or~\eqref{llProblem} can be solved either optimally or greedily (i.e., by stopping the exact solution process at the first solution found for~\eqref{eq:NewDeltayHat_1}, \eqref{eq:NewDeltayHat_2} and the first solution with better objective function value than $y^*$ for~\eqref{llProblem}). We denote these variants by the additional letters \texttt{O} and \texttt{G}, respectively. Regarding the greedy variant, we note that any feasible solution to~\eqref{eq:NewDeltayHat_1} or~\eqref{eq:NewDeltayHat_2} yields an improving direction. When solving~\eqref{llProblem} greedily in \texttt{FG}, we impose an upper cutoff strictly smaller than $q(y^*)$ on the objective function value to ensure that the obtained solution is improving.
We also emphasize that, following preliminary experiments and the findings reported in \cite{Gaar-et-al:2024}, we do the DC separation procedure for both integer and fractional solutions $(x^*,y^*)$ encountered during the \BC, even if for correctness we only need the separation procedure for integer solutions. In particular, at the root node of the \BC tree we continue adding DCs until no more significantly violated cuts can be found. Note that we employ a normalized violation threshold of $10^{-5}$ throughout the algorithm. At all other \BC nodes, we limit the number of separation rounds to one for fractional solutions (and for integer solutions we of course have no limit to ensure correctness).

\subsection{Identifying redundant disjuncts}
\label{sec:redundancy_check}
Next, we describe how we deal with disjuncts that are redundant for our chosen $\Delta\hat y$ due to results in Section~\ref{section:redundantDisjunctions}.
In our baseline simple setting \texttt{RS}, we employ a small redundancy-check procedure based on Lemmas~\ref{lemma:Fischetti17redundant} and~\ref{lemma:unifiedRedundancy}:
\begin{itemize}
\item If $B_i \Delta \hat y \ge 0$, then the associated lower-level disjunct $\cD_i(\Delta \hat y)$ is redundant by Lemma~\ref{lemma:Fischetti17redundant}.
\item If $\sum_{j=1}^{n_1} \min\{ A_{ij} x^-_{j}, A_{ij} x^+_{j}\}  + \sum_{j=1}^{n_2} \min\{ B_{ij} y^-_j, B_{ij} y^+_j\} \ge f_i - B_i \Delta \hat y$, then the associated lower-level disjunct $\cD_i(\Delta \hat y)$ is redundant by Lemma~\ref{lemma:unifiedRedundancy}.\ref{lemma:unifiedRedundancy:linking}.
\item If $\Delta \hat y_j \le u_j-y^+_j$, then the associated upper-bound disjunct $\cD_j^u(\Delta \hat y)$ is redundant by Lemma~\ref{lemma:unifiedRedundancy}.\ref{lemma:unifiedRedundancy:bounds}.
\item If $\Delta \hat y_j \ge \ell_j-y^-_j$, then the associated lower-bound disjunct $\cD_j^\ell(\Delta \hat y)$ is redundant by Lemma~\ref{lemma:unifiedRedundancy}.\ref{lemma:unifiedRedundancy:bounds}.
\end{itemize}
These checks can be performed efficiently by simple inspection. 

We have also implemented a more time-consuming large redundancy check, namely the optimality-based removal strategy based on Lemma~\ref{lemma:RO}. This strategy corresponds to the strategy \texttt{RO} in \cite{Gaar-et-al:2024} and we denote it as \texttt{RL} in our work. We note that this strategy involves solving an ILP for every disjunct which we want to check. As already discussed above, in contrast to \cite{Gaar-et-al:2024}, due to the bilevel-free set we are using, we have many more disjuncts, namely one more for each upper-bound and lower-bound constraint on $y$ and one more for each lower-level constraint that is not a linking constraint. Thus directly using this \texttt{RL}-strategy is very time consuming, even if we only check disjuncts using Lemma~\ref{lemma:RO} when the simple strategies failed for this disjunct. 

For this reason, we also consider the medium strategy \texttt{RM} that is less computationally demanding than \texttt{RL}. 
In this strategy, only the objective disjunct $\mathcal{D}_0(\Delta \hat y)$ and the lower-level disjuncts $\mathcal{D}_i(\Delta \hat y)$ for $i \in \{1,\dots,m_2\}$ are tested for redundancy based on Lemma~\ref{lemma:RO} via solving ILPs, while redundancy checks for the lower-bound disjuncts $\mathcal{D}_j^\ell(\Delta \hat y)$ and the upper-bound disjuncts $\mathcal{D}_j^u(\Delta \hat y)$ are skipped. This strategy is motivated by preliminary computations, which showed that for the lower-bound and upper-bound disjuncts the more time-consuming check via ILP rarely determined the disjuncts to be redundant, i.e., the simple checks had nearly the same detection rate as the check via ILP. We provide some results about this in Section~\ref{section:convexResults}.

\subsection{Constructing and solving the CGLP}
\label{sec:implCGLP}

After performing the redundancy tests for the disjuncts, the next step is to build and solve the CGLP, which is based on the current \BC node relaxation $\mathcal{P}$ described in the beginning of Section~\ref{sec:BC}. In particular,~$\mathcal{P}$ can contain branching decisions at the current node and DCs added at ancestor nodes by incorporating them into $P$, $Q$ and $r$ within~\eqref{previousCuts} when defining $\mathcal{P}$.
We note that, although including the DCs added at ancestor nodes in the CGLP through $\mathcal{P}$ is theoretically required to guarantee successful separation of bilevel-infeasible integer solutions, it is not mandatory for purely binary problems. In the latter case, any integer solution of the node LP is an extreme point of $\mathcal{P}$, independent of whether $\mathcal{P}$ is defined with or without the DCs added at ancestor nodes.
Conversely, for instances with general integer variables, it is in principle possible that a bilevel-infeasible integer solution $(x^*,y^*)$ fails to be an extreme point of $\mathcal{P}$ when previous DCs are not added to~$\mathcal{P}$. In such a situation, the CGLP may not yield a DC that is violated at $(x^*,y^*)$. 
Preliminary computations indicated that, although incorporating previous DCs into $\mathcal{P}$ is numerically more robust and sometimes yields stronger DCs, it may adversely affect the overall performance due to the need to construct and solve substantially larger CGLPs.
To balance these considerations, we adopt the following strategy: 
\begin{itemize}
    \item All previously generated DCs are included in $\mathcal{P}$ in the CGLP at the root node of the B\&C tree. 
    \item At other nodes with bilevel-infeasible solutions (recall that this means the solutions are integer), we first try separation via solving the CGLP without adding the previously generated DCs to $\mathcal{P}$. Only if this separation fails and does not produce a violated DC, we call the separation again, but this time with the previously generated DCs added to $\mathcal{P}$.
    \item At non-root nodes with fractional solutions, we try separation via solving the CGLP without adding the previously generated DCs to $\mathcal{P}$. If this separation fails, we continue branching, unless we use option \texttt{H} and employ an improving-direction pool which contains other potential improving directions, in which case we proceed to the next feasible direction of the pool and try separation again.
\end{itemize}

After having determined $\mathcal{P}$, we are now ready to describe the details of the construction of the CGLP.
For each disjunct that is identified as redundant, we remove the corresponding variables and constraints from the CGLP. Thus,  we construct \eqref{eq:CGLP} using only the non-redundant disjuncts in $\mathcal{K}$ and have to solve smaller LPs for deriving DCs.

Note that in the following special cases, we skip constructing and solving the CGLP entirely:

\begin{itemize}
    \item Every disjunct is found to be redundant. In this situation, we prune the \BC node by inserting a dummy violated inequality ($0x+0y\leq -1$).
    \item All disjuncts except for the objective disjunct $\mathcal{D}_0(\Delta \hat y)$ are redundant. In that case, we directly introduce its inequality $ 2{\Delta \hat y} ^\top R y \geq - \Delta \hat y ^\top R\Delta \hat y - g^\top \Delta  \hat y$ into the \BC node relaxation $\mathcal P$.
\end{itemize}

Finally, we want to briefly mention that in our \BC, we also try to separate integer solutions that are produced by the \BC MILP solver's primal heuristics. Naturally, such solutions are not necessarily extreme points of $\mathcal{P}$ when considering integer (and not binary) instances, thus we are not guaranteed to find a violated \DC for them. If our separation fails to produce a violated \DC for such a solution even after including previously generated \DCs in $\mathcal{P}$, we reject it as an incumbent, to preserve the correctness of the \BC.

\subsection{Maintaining the improving-direction pool for heuristic DC generation}
\label{sec:implPool}

In our computational study, an additional suffix \texttt{H} in a setting indicates an improving-direction pool $\mathcal{C}$ is employed for determining the improving direction $\Delta \hat y$ for a given $(x^*, y^*)$ heuristically. 
This pool $\mathcal{C}$ is initialized with $\pm 1$ unit vectors of length $n_2$. Every time we solve an IDG problem, the pool is (possibly) expanded by the new improving directions $\Delta \hat y$ found while solving the corresponding MINLP, i.e.,~\eqref{eq:NewDeltayHat_1},~\eqref{eq:NewDeltayHat_2} and~\eqref{llProblem} for \texttt{S}, \texttt{B} and \texttt{F}, respectively. In particular, if the IDG problem is solved optimally in setting \texttt{O}, all the improving directions found during the solution procedure are used to expand the pool. Otherwise, if they are solved greedily in setting \texttt{G}, the first improving direction found feeds the pool.

We use a capacitated pool having at most $C$ elements. Once the capacity $C$ is reached, the improving directions with the lowest utilization percentages (ratio of how many times this  improving direction has produced a violated DC so far to the total number of trials) are replaced by the newly introduced ones. 
Moreover, we employ a diversification strategy for selecting directions from the pool $\mathcal{C}$ as $\Delta \hat y$: each time the heuristic IDG is called, the search begins from the improving direction immediately following the one that was most recently used to generate a DC. This proved helpful in preliminary computations, compared to e.g., always iterating through the pool from the start.

For separating fractional solutions only the pool-based heuristic separation is executed without being followed by the computationally more costly exact separation. The only exception is the root node, at which we perform also exact separation once the pool has been exhaustively explored, even for a fractional solution. 

An additional strategy we examine is the simultaneous addition of multiple DCs derived from different~$\Delta \hat y$ obtained within IDG for a single $(x^*,y^*)$ at once, which will be denoted by the suffix \texttt{M}. Based on an initial analysis, we propose the following rules to be applied when the multiple-DC strategy is enabled: $(i)$ For the \texttt{SO}, \texttt{BO} and \texttt{FO} configurations, each time we solve the corresponding IDG problem, we introduce one DC for every distinct improving direction associated with the solutions stored by the MINLP solver; $(ii)$ If the heuristic IDG option \texttt{H} is also enabled, the scanning of the pool is stopped once $\texttt{maxIntHeurCuts}=10$ cuts have been added at integer nodes, and $\texttt{maxFracHeurCuts}=5$ cuts have been added at non-root fractional nodes. At the root node of the \BC tree, no such bounds are enforced.

%%%%%%%%%%%%%

\section{Computational study}\label{sec:results}

In the following parts, we evaluate our \BC framework on 640 test instances across seven problem classes. The implementation was done in C++ using CPLEX 22.1 as the MILP and MINLP solver, with its optimality target parameter set to global optimality when solving MINLPs that include non-convex quadratic terms arising in IDG. All settings of CPLEX were left on their default values, except the feasibility tolerance parameter which is reduced from $10^{-6}$ to $10^{-8}$ while solving \eqref{eq:CGLP}, to mitigate numerical difficulties. The computations were performed on a single core of an Intel Xeon X5570 machine with 2.93 GHz processor and~8~GB of RAM, and the time limit for each run was set to 600~seconds.
We first introduce the testbed in Section~\ref{sec:instances}, tune key algorithmic components like IDG strategies, redundancy checks and DC generation on instances with convex follower objective function, and compare our best configurations against state-of-the-art methods from the literature in Section~\ref{section:convexResults}. Finally, we assess performance on instances with non-convex follower objective functions in Section~\ref{sec:CompNonConv}.

%%%%%%%%%%%%%
\subsection{Test instances}
\label{sec:instances}

Our experimental testbed comprises 640 test instances partitioned into seven distinct problem sets, incorporating both established benchmark sets and new problem classes. Table~\ref{tab:instances} summarizes the main characteristics of all instance sets. All instances can be accessed publicly at \url{https://msinnl.github.io/pages/instancescodes.html}\footnote{Instances will be uploaded upon publication of the paper.}.

\begin{table}[htb]
    \centering \small
    \caption{Summary of characteristics of test instance sets}
    \label{tab:instances}
    \begin{tabular}{lllrr}
    \hline
    Instance set & $q(y)$ & Variable type & $m_2$ & $\#$ \\
    \hline
    \texttt{QBCov-1}             & Convex     & Binary  & 1         & 40  \\
    \texttt{QBCov-2}             & Convex     & Binary  & 2         & 40  \\
    \texttt{QBMKP-1}   & Convex     & Binary  & 1         & 100 \\
    \texttt{QBMKP-2} & Convex     & Binary  & 2         & 100 \\
    \texttt{QBMKPInt}& Convex     & Integer & 1         & 100 \\
    \texttt{QBCovMulti}          & Convex     & Binary  & $\in\{3,4\}$ & 80  \\   
    \texttt{BNQKP}                & Non-convex & Integer & 1         & 180 \\
    \hline
    \end{tabular}
\end{table}

\paragraph{Main benchmark instances (\texttt{QBCov} and \texttt{QBMKP})} 
In our main benchmark, we consider instances with a convex follower objective function ($q(y) = y^\top R y$ with positive semidefinite $R$ and $g=\mathbf{0}$), as this problem class already was studied previously. We use five instance sets also considered in \cite{Gaar-et-al:2024} with 380 instances in total. The \texttt{QBCov} test instance sets comprise 80 binary covering instances with $n = n_1 + n_2 \in \{20, 30, 40, 50\}$ binary leader and follower variables, $m_1 \in \{0,1\}$ pure leader constraints, and $m_2 \in \{1, 2\}$ lower-level constraints (separated in \texttt{QBCov-1} and \texttt{QBCov-2}), which are all linking constraints. The \texttt{QBMKP} family consists of 300 instances derived from 0--1 multiple knapsack instances in the SAC-94 library~\citep{sac94}, with 10 to 105 items partitioned between leader and follower. This family includes two sets with binary leader and follower variables \texttt{QBMKP-1} and \texttt{QBMKP-2} with $m_2=1$ and $m_2=2$, respectively, as well as one set \texttt{QBMKPInt}, where integer leader and follower variables take values in $\{0, \dots, 5\}$ and $m_2 = 1$. Also in this set, all lower-level constraints are linking constraints. We refer the reader to \cite{Gaar-et-al:2024} for details regarding their generation.

\paragraph{Multiple linking constraints extension instances (\texttt{QBCovMulti})} 
To evaluate algorithmic behavior under denser linking structures, we generate \texttt{QBCovMulti}, an additional set of 80 binary covering instances. These instances extend the \texttt{QBCov} family by incorporating multiple ($m_2 \in \{3, 4\}$) lower-level constraints, which are all linking constraints, while maintaining a convex follower objective function and binary leader and follower variables.

\paragraph{Bilevel non-convex quadratic knapsack problem (BNQKP) instances (\texttt{BNQKP})}
Finally, we construct the set \texttt{BNQKP} of 180 bilevel quadratic 0--1 knapsack problem instances featuring a non-convex follower objective function. This instance set is derived from the nonlinear bilevel knapsack problem instances with quadratic leader objective function in the work of \cite{BeheshtiOZP15}, following the approach of \cite{zenarosa2021exact} and have the general BNQKP form
\begin{align*}
\max_{x\in \mathbb{Z}} \quad & c_u y- x \\
\st \quad & \underline{x} \le x \le \overline{x} \\
                  & y \in \arg\max_{y' \in \{0,1\}^{n_2}} \left\{ y'^\top Q y' + d_\ell^\top y' : w^\top y' \le x \right\}.
\end{align*} 

In our experimental study, we consider a subset of the generated BNQKP instances defined by knapsack sizes \(n_2 \in \{25, 50\}\) and a non-diagonal profit matrix \(Q\). This subset comprises 180 instances that differ in at least one of the following: the density of the off-diagonal entries of $Q$, the correlation between item weights and the follower’s profits, and the maximum magnitude of the quadratic coefficients. We refer the reader to the works of \cite{BeheshtiOZP15} and \cite{zenarosa2021exact} for further details.

%%%%%%%%%%%%%%%%%%%

\subsection{Results for instances with convex follower objective function}
\label{section:convexResults}

In this section, we first present numerical results regarding fine tuning our algorithmic settings and then provide the results of a numerical study where we examine the performance of the DC based B\&C method we propose against a state-of-the-art method from literature. All instances addressed in this section involve a convex quadratic follower objective function.

\paragraph{Tuning the algorithmic settings}
\citet{Gaar-et-al:2024} showed that solving the follower problem greedily is a more efficient strategy than solving it to optimality. In our preliminary study, for each of the IDG strategies, we examine the combined effect of $(i)$ solving the IDG problem either optimally or greedily, and $(ii)$ the level of effort devoted to redundancy checks. These aspects may be tightly linked, particularly in the \texttt{S} and~\texttt{B} settings for IDG, where the objective of the corresponding MINLP is to maximize (via an $L_1$-surrogate, or exactly, respectively) the number of redundant disjuncts. When these MINLPs are solved greedily (i.e., not necessarily optimally), additional redundancy checks can either be beneficial or end up more expensive than solving the MINLP optimally. To investigate this interaction, we compare the greedy \texttt{G} and optimal \texttt{O} variants for each of the three IDG strategies \texttt{S}, \texttt{B} and \texttt{F}, under the three redundancy-check configurations \texttt{RS}, \texttt{RM}, and \texttt{RL}.

We present in Table~\ref{tab:redundancy_check} the average results over 414 (out of 460) instances that are identified as feasible, in terms of the following measures: 
the number of instances solved to optimality out of 414 (\texttt{nSol}),
the final percent optimality gap (\texttt{Gap}),
the runtime in seconds ($t$),
the time spent for redundancy check ($t_R$), 
the percentage of the checks that resulted in removing the disjunct (\texttt{red}(\%)),
the percentage of the ILP-based checks that resulted in removing the disjunct ($\texttt{red}_{\texttt{RL}}(\%)$),
the number of \BC nodes (\texttt{nBCN}), 
the number of DCs separated at integer solutions (\texttt{nICut}), and 
the number of DCs separated at fractional solutions (\texttt{nFCut}). The gap is computed as $\texttt{Gap} = 100 \cdot (\texttt{UB} - \texttt{LB}) / |\texttt{UB}|$, where \texttt{UB} and \texttt{LB} denote the incumbent value and the dual bound at termination, respectively, and is set to $100\%$ if no bilevel-feasible solution is found.

% Table generated by Excel2LaTeX from sheet 'summary'
\begin{table}[tbh]
  \centering \footnotesize
  \caption{Results over 414 feasible instances}
\begin{tabular}{rlcrrrrrrrr}
\hline
\multicolumn{1}{l}{\texttt{Setting}} & \multicolumn{1}{c}{\texttt{Red.\,check}} & \multicolumn{1}{c}{\texttt{nSol}} & \multicolumn{1}{c}{\texttt{Gap}} & \multicolumn{1}{c}{$t$} & \multicolumn{1}{c}{$t_R$} & \multicolumn{1}{c}{\texttt{red}($\%$)} & \multicolumn{1}{c}{$\texttt{red}_{\texttt{RL}}(\%)$} & \multicolumn{1}{c}{\texttt{nBCN}} & \multicolumn{1}{c}{\texttt{nICut}} & \multicolumn{1}{c}{\texttt{nFCut}} \\
\hline
\multicolumn{1}{l}{\texttt{FO}} & \texttt{RS}    & 235   & 26.4  & 295.3 & 0.0   & 84.4  & 0.0   & 3176.9 & 262.1 & 2388.9 \\
      & \texttt{RM}   & 243   & 26.6  & 285.2 & 8.7   & 85.8  & 32.1  & 2748.3 & 261.8 & 2059.7 \\
      & \texttt{RL}    & 241   & 26.7  & 286.4 & 20.5  & 86.4  & 20.8  & 2495.5 & 251.4 & 1889.1 \\
\hline
\multicolumn{1}{l}{\texttt{FG}} & \texttt{RS}    & 290   & 20.0  & 211.5 & 0.0   & 85.1  & 0.0   & 5953.0 & 595.6 & 4677.1 \\
      & \texttt{RM}   & 294   & 19.9  & 207.8 & 17.0  & 86.2  & 29.1  & 4998.3 & 603.5 & 3939.3 \\
      & \texttt{RL}    & 290   & 20.6  & 211.4 & 37.7  & 86.7  & 19.1  & 4327.1 & 566.1 & 3437.1 \\
      \hline
\multicolumn{1}{l}{\texttt{S}\texttt{O}} & \texttt{RS}    & 361   & 8.5   & 104.9 & 0.0   & 90.9  & 0.0   & 3602.4 & 64.1  & 3096.7 \\
      & \texttt{RM}   & 360   & 8.8   & 103.0 & 7.1   & 91.5  & 16.4  & 3130.1 & 58.0  & 2695.2 \\
      & \texttt{RL}    & 361   & 8.7   & 102.3 & 9.6   & 91.8  & 15.1  & 2988.5 & 56.5  & 2568.7 \\
\hline
\multicolumn{1}{l}{\texttt{S}\texttt{G}} & \texttt{RS}    & 350   & 10.0  & 108.7 & 0.0   & 90.5  & 0.0   & 8148.2 & 147.4 & 7236.2 \\
      & \texttt{RM}   & 355   & 8.5   & 105.3 & 25.1  & 91.4  & 20.5  & 5893.3 & 101.2 & 5169.9 \\
      & \texttt{RL}    & 354   & 8.3   & 106.0 & 30.5  & 91.6  & 18.9  & 5509.1 & 96.3  & 4824.5 \\
\hline
\multicolumn{1}{l}{\texttt{B}\texttt{O}} & \texttt{RS}    & 352   & 9.9   & 119.7 & 0.0   & 91.2  & 0.0   & 3223.2 & 71.3  & 2760.2 \\
      & \texttt{RM}   & 359   & 8.9   & 112.1 & 6.4   & 92.0  & 18.7  & 2592.4 & 52.7  & 2204.0 \\
      & \texttt{RL}    & 360   & 9.0   & 112.6 & 8.2   & 92.2  & 17.6  & 2552.4 & 52.8  & 2171.5 \\
\hline
\multicolumn{1}{l}{\texttt{B}\texttt{G}} & \texttt{RS}    & 349   & 10.6  & 111.5 & 0.0   & 90.4  & 0.0   & 6799.6 & 122.1 & 5982.9 \\
      & \texttt{RM}   & 355   & 9.3   & 107.0 & 18.4  & 91.2  & 19.8  & 5024.2 & 88.9  & 4369.1 \\
      & \texttt{RL}    & 357   & 9.4   & 108.5 & 25.6  & 91.5  & 18.0  & 4650.7 & 84.9  & 4025.8 \\
\hline
\end{tabular}%
  \label{tab:redundancy_check}%
\end{table}%

Table~\ref{tab:redundancy_check} clearly illustrates that the \texttt{F} configurations, especially $\texttt{FO}$, perform very poorly, with significantly fewer instances optimally solved within the time limit than all other IDG strategies, under all redundancy check strategies. From column \texttt{red}(\%) we understand that optimal follower solutions do not lead to more redundancies compared to greedy solutions. Under \texttt{S} and \texttt{B} configurations the redundancy ratio is higher, and most of the redundancies already stem from solving~\eqref{eq:NewDeltayHat_1} and~\eqref{eq:NewDeltayHat_2} that determine $\Delta \hat y$ to maximize redundancy, which is supported by higher \texttt{red}(\%) and lower $\texttt{red}_{\texttt{RL}}(\%)$ values. 
As the redundancy checking effort increases, both the number of generated DCs and the number of \BC nodes decrease significantly, while the number of instances solved to optimality remains largely unchanged for \texttt{S} and \texttt{B}. The most aggressive setting, \texttt{RL}, attains good final gaps with substantially fewer DCs, suggesting that the cuts become stronger even when only a relatively small number of additional nontrivial redundancies are detected. Based on these experiments, we designate \texttt{SO} and \texttt{BO}, along with the \texttt{RL} option, as the most reliable configurations. We therefore use these two settings in the subsequent experiments and analyze their performance when combined with additional components.

Next, we investigate the impact of the heuristic improving-direction pool approach to cut generation, and of adding multiple DCs for the same bilevel-infeasible solution, whenever possible.
In Table~\ref{tab:heuristic_multi} we report the following additional measures: the number of DCs separated at the root node (\texttt{nRCut}), which are also counted within \texttt{nICut} and \texttt{nFCut}, the final gap with respect to the best-known upper-bound (\texttt{Gap*}), the root gap (\texttt{RGap}), the root gap with respect to the best-known upper-bound (\texttt{RGap*}), the time needed to solve the IDG problem ($t_\Delta$), the time spent for checking disjunct redundancy ($t_R$), and the additional time needed to solve the separation problem \eqref{eq:CGLP} ($t_S$). \texttt{RGap} is computed as \texttt{Gap} using the incumbent value and the dual bound at the root node, and the starred measures use the best-known upper bound instead of the incumbent value.

Notably in Table~\ref{tab:heuristic_multi}, \texttt{Gap} and \texttt{RGap} are substantially larger than their starred counterparts (\texttt{Gap*} and \texttt{RGap*}) because they are computed solely with respect to the incumbent of the current run, whereas the starred measures capture the actual quality of the dual bound relative to the best-known upper-bound.
From an algorithmic perspective, the pool-based heuristic separation procedure (\texttt{H}) dramatically lowers IDG problem solution times ($t_\Delta$) by avoiding the need to always solve~\eqref{eq:NewDeltayHat_1} or~\eqref{eq:NewDeltayHat_2}. Since many potential improving directions are available, we permit aggressive cut separation at the root node, which raises \texttt{nRCut}. In the \texttt{M} variant, the multiple-DC separation procedures strengthens the node LP relaxation much more aggressively, achieving over a 70\% decrease in the number of \BC nodes (\texttt{nBCN}) for both \texttt{SO} and \texttt{BO}. Although \texttt{M} solves one fewer instance with \texttt{SO}, it solves four more instances with \texttt{BO}. To have both IDG strategies under identical settings, we keep \texttt{M} and consider \texttt{SO-RL-H-M} and \texttt{BO-RL-H-M} as our best configurations, which we henceforth denote by \texttt{S+} and \texttt{B+}, respectively.
%

% Table generated by Excel2LaTeX from sheet 'summary'
\begin{table}[tbh]
  \centering \footnotesize
  \caption{Results over 414 feasible instances - heuristic cut generation and multiple cuts}
    \begin{tabular}{lrrrrrrrrrrrrr}
    \hline
\multicolumn{1}{l}{\texttt{Setting}} &\multicolumn{1}{c}{\texttt{nSol}} & \multicolumn{1}{c}{$t$} & \multicolumn{1}{c}{\texttt{Gap}} & \multicolumn{1}{c}{\texttt{RGap}} & \multicolumn{1}{c}{\texttt{Gap*}}  & \multicolumn{1}{c}{\texttt{RGap*}} & \multicolumn{1}{c}{\texttt{nBCN}} & \multicolumn{1}{c}
{\texttt{nRCut}} & \multicolumn{1}{c}
{\texttt{nICut}} & \multicolumn{1}{c}{\texttt{nFCut}} & \multicolumn{1}{c}{$t_\Delta$} & \multicolumn{1}{c}{$t_S$} & \multicolumn{1}{c}{\texttt{$t_R$}} \\
\hline 
\texttt{S}\texttt{O-RL} & 361   & 102.3 & 8.7   & 69.7  & 4.2   & 15.1  & 2988.5 & 70.2  & 56.5  & 2568.7 & 72.5  & 16.1  & 9.6\\
\texttt{S}\texttt{O-RL-H} & 367   & 93.2  & 8.9   & 69.0  & 4.0   & 14.8  & 3176.4 & 151.8 & 55.1  & 2837.9 & 19.2  & 53.6  & 15.7\\
\texttt{S}\texttt{O-RL-H-M} & 366   & 96.1  & 9.6   & 71.0  & 4.0   & 15.5  & 862.7 & 158.5 & 54.5  & 2364.6 & 22.1  & 55.5  & 16.5\\
\texttt{B}\texttt{O-RL} & 360   & 112.6 & 9.0   & 70.3  & 4.2   & 15.4  & 2552.4 & 73.2  & 52.8  & 2171.5 & 87.6  & 13.5  & 8.2\\
\texttt{B}\texttt{O-RL-H} & 363   & 98.6  & 9.2   & 69.9  & 4.1   & 15.4  & 3282.9 & 148.2 & 55.9  & 2911.8 & 23.3  & 55.4  & 15.9\\
\texttt{B}\texttt{O-RL-H-M} & 367   & 99.7  & 8.6   & 70.2  & 4.1   & 15.5  & 914.4 & 144.7 & 40.1  & 2562.1 & 26.7  & 53.3  & 17.5\\
\hline
\end{tabular}%
  \label{tab:heuristic_multi}%
\end{table}%

\paragraph{Comparison with a state-of-the-art method}

Lastly, we conduct a performance analysis with respect to the benchmark method \texttt{SOCP} proposed by \cite{Gaar-et-al:2024}. To this end, we focus on the two best-performing configurations \texttt{S+} and \texttt{B+}. The results are first illustrated by cumulative distribution graphs of solution times and optimality gaps in Figure~\ref{fig:overall_comparison} over all instance sets (excluding \texttt{QBCovMulti}, for which no ready-to-use implementation of \texttt{SOCP} is available), and then reported separately for each instance set in Table~\ref{tab:SOCP}. In both, we restrict our attention to the instances that are proven to be feasible by at least one of the three methods.
 The time profile in Figure~\ref{fig:overall_comparison} shows that \texttt{S+} and \texttt{B+} solve approximately $75\%$ of instances within the first 50 seconds and plateau near $85\%$ by the time limit, whereas \texttt{SOCP} solves around 
70\% of instances optimally within the time limit. Instances for which no feasible solution is found within the time limit are assigned a gap of $100\%$, accounting for the sharp jump at the right end of the gap profile. Since most runs have no incumbent at the root node, the root gap profile is based on \texttt{RGap*}, which measures the strength of the root bound independently of the root incumbent. Across all three criteria, \texttt{S+} and \texttt{B+} perform nearly identically with no consistent dominance of one over the other.
\begin{figure}[tbh]
    \centering
    \includegraphics[width=0.45\linewidth]{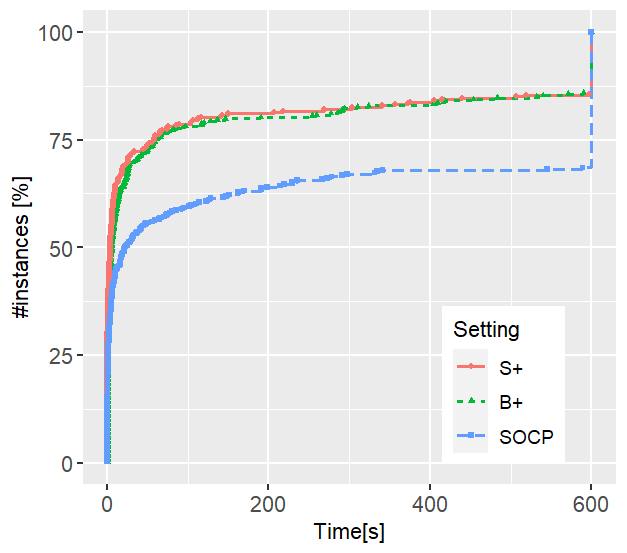}
\includegraphics[width=0.45\linewidth]{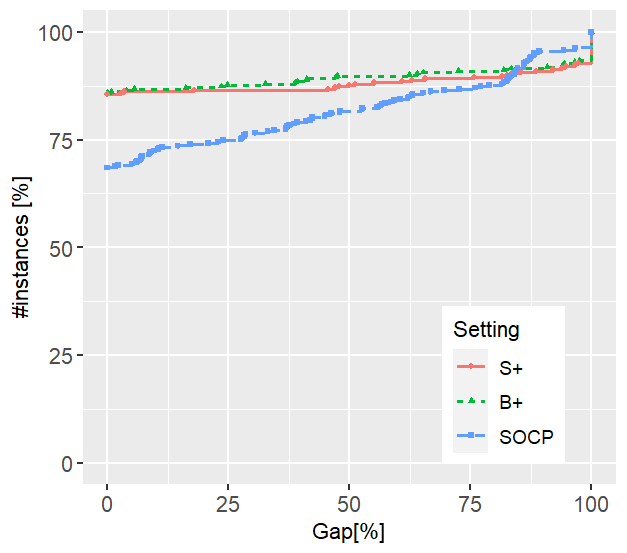}
\includegraphics[width=0.45\linewidth]{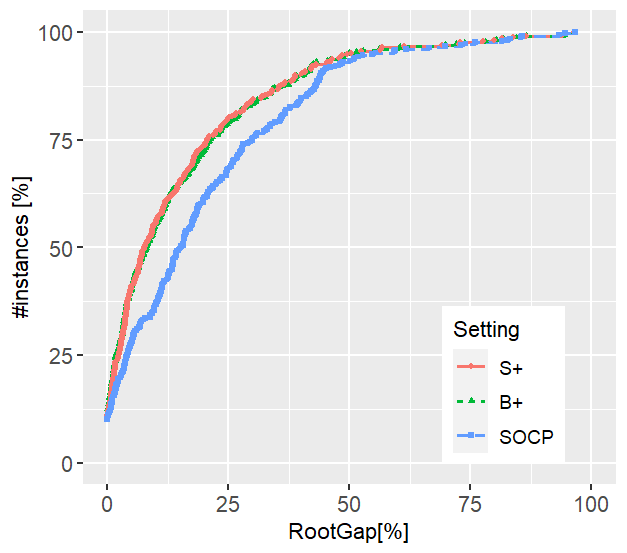}
    \caption{Our best settings compared to the state-of-the-art method}
    \label{fig:overall_comparison}
\end{figure}

% Table generated by Excel2LaTeX from sheet 'summary'
\begin{table}[tbh]
    \setlength{\tabcolsep}{4pt}
  \centering \footnotesize
  \caption{Our best settings compared to the state-of-the-art method: average results over 334 feasible instances}
    \begin{tabular}{lllrrrrrrrrrrrr}
    \hline
\texttt{Instances} & \multicolumn{1}{c}{\texttt{Setting}} & \multicolumn{1}{c}{\texttt{nSol}} & \multicolumn{1}{c}{$t$} & \multicolumn{1}{c}{\texttt{Gap}} & \multicolumn{1}{c}{\texttt{RGap}} & \multicolumn{1}{c}{\texttt{Gap*}}  & \multicolumn{1}{c}{\texttt{RGap*}} & \multicolumn{1}{c}{\texttt{nBCN}} & \multicolumn{1}{c}
{\texttt{nRCut}} & \multicolumn{1}{c}
{\texttt{nICut}} & \multicolumn{1}{c}{\texttt{nFCut}} & \multicolumn{1}{c}{$t_\Delta$} & \multicolumn{1}{c}{$t_S$} & \multicolumn{1}{c}{\texttt{$t_R$}} \\
\hline
    \multicolumn{1}{l}{\texttt{QBCov-1}} & \texttt{SOCP}  & 38/40 & 67.5  & 1.2   & 80.8  & 0.8   & 20.5  & 6989.8 & 13.9  & 278.4 & 4382.6 & 12.8  & 40.3  & 11.3 \\
          & \texttt{S+}   & 40/40 & 2.6   & 0.0   & 55.8  & 0.0   & 9.0   & 78.2  & 16.1  & 30.1  & 36.5  & 0.9   & 0.2   & 0.2 \\
          &\texttt{B+} & 40/40 & 4.2   & 0.0   & 53.5  & 0.0   & 8.9   & 105.5 & 13.8  & 25.6  & 63.8  & 2.2   & 0.2   & 0.2 \\
    \hline
    \multicolumn{1}{l}{\texttt{QBCov-2}} & \texttt{SOCP}  & 35/40 & 132.3 & 4.2   & 79.2  & 2.7   & 24.9  & 9143.4 & 14.5  & 427.4 & 6052.8 & 18.4  & 78.7  & 31.1 \\
          & \texttt{S+}   & 40/40 & 8.9   & 0.0   & 60.4  & 0.0   & 15.3  & 385.3 & 17.6  & 58.2  & 375.2 & 5.0   & 1.3   & 1.6 \\
          & \texttt{B+} & 40/40 & 12.0  & 0.0   & 65.6  & 0.0   & 15.2  & 389.5 & 16.1  & 51.3  & 421.9 & 8.0   & 1.5   & 1.7 \\
    \hline
    \multicolumn{1}{l}{\texttt{QBMKP-1}} & \texttt{SOCP}  & 54/70 & 157.5 & 14.4  & 67.8  & 6.2   & 16.8  & 12125.7 & 22.5  & 426.3 & 9762.9 & 24.8 & 92.9  & 31.9 \\
          & \texttt{S+}   & 66/70 & 62.0  & 3.6   & 67.5  & 1.8   & 16.2  & 1170.8 & 27.2  & 28.1  & 3063.8 & 14.3  & 26.7  & 19.9 \\
          & \texttt{B+} & 67/70 & 64.4  & 2.8   & 66.5  & 1.8   & 16.0  & 1141.9 & 21.1  & 25.5  & 2941.3 & 19.8  & 24.4  & 19.2 \\
    \hline
    \multicolumn{1}{l}{\texttt{QBMKP-2}} & \texttt{SOCP}  & 62/85 & 188.6 & 12.6  & 65.7  & 5.6   & 13.6  & 9852.4 & 20.7  & 409.1 & 8151.7 & 22.6 & 113.4 & 46.6 \\
          & \texttt{S+}   & 76/85 & 97.8  & 7.2   & 61.0  & 1.9   & 12.9  & 1123.0 & 22.3  & 46.8  & 2849.1 & 54.1  & 21.8  & 20.7 \\
          & \texttt{B+} & 75/85 & 100.0 & 6.5   & 63.8  & 1.9   & 12.9  & 1165.0 & 25.1  & 36.3  & 2931.5 & 56.6  & 22.2  & 20.1 \\
    \hline
    \multicolumn{1}{l}{\texttt{QBMKPInt}} & \texttt{SOCP}  & 40/99 & 375.3 & 37.9  & 89.5  & 18.0  & 25.9  & 27001.2 & 50.8  & 704.7 & 24591.7 & 56.2 & 228.2 & 67.8 \\
          & \texttt{S+}   & 64/99 & 249.0 & 31.4  & 86.5  & 13.8  & 18.5  & 1050.3 & 595.1 & 26.3  & 4416.9 & 23.6  & 189.3 & 31.2 \\
          & \texttt{B+} & 65/99 & 259.2 & 28.4  & 84.8  & 14.1  & 18.8  & 1276.6 & 541.4 & 20.2  & 5255.0 & 34.5  & 182.4 & 37.1 \\
    \hline
    \end{tabular}%
  \label{tab:SOCP}%
\end{table}%

As shown in Table~\ref{tab:SOCP}, both proposed configurations (\texttt{S+} and \texttt{B+}) consistently outperform \texttt{SOCP} across all five instance sets in terms of number of optimally solved instances, execution speed, and \BC tree size. On the \texttt{QBCov-1} and \texttt{QBCov-2} sets, both methods achieve a $100\%$ solve rate ($40/40$) with an average runtime of at most 12 seconds, whereas \texttt{SOCP} fails to solve several instances, and requires significantly longer runtimes. On the more challenging \texttt{QBMKP} sets, the advantages of our approach become even more pronounced. For instance, on \texttt{QBMKPInt}, \texttt{S+} and \texttt{B+} solve $64$ and $65$ out of $99$ instances, respectively, compared to only $40$ by \texttt{SOCP}. In general, \texttt{SOCP} separates many more cuts and leads to significantly larger \BC trees.
Finally, \texttt{S+} and \texttt{B+} exhibit highly comparable performance overall, with \texttt{S+} holding a slight runtime advantage on the \texttt{QBCov} instances.

\subsection{Results for instances with non-convex follower objective function}
\label{sec:CompNonConv}

For the BNQKP, the non-convex follower objective function turns the IDG problems~\eqref{eq:NewDeltayHat_1} and~\eqref{eq:NewDeltayHat_2} of \texttt{S} and \texttt{B} into computationally heavy non-convex MINLPs. Conversely, the  IDG strategy \texttt{F} reduces to a pure 0-1 quadratic integer program (QIP) with non-convexity restricted to the objective function. 
Furthermore, because BNQKP lacks leader constraints~\eqref{MxNy}, solving the follower problem to optimality in \texttt{FO} guarantees a bilevel-feasible solution that can immediately serve as an incumbent. In contrast, although the greedy variant \texttt{FG} achieved optimal solutions with a similar frequency in preliminary experiments not shown here, it frequently fails to produce valid incumbents for instances that terminate by timeout. This leads to misleadingly biased average upper-bounds (\texttt{UB}) and gaps. To ensure a consistent and reliable comparative evaluation, we therefore adopt \texttt{FO} as our main configuration for the current instance type, along with the \texttt{RM} strategy for redundancy checks based on the results in Table~\ref{tab:redundancy_check}. The heuristic cut generation \texttt{H} and the addition of multiple cuts \texttt{M} are likewise activated.

We evaluate \texttt{FO} against another benchmark method, \texttt{MIX++} \citep{fischetti2017new}, across 180 BNQKP instances disaggregated over three matrix density levels $p \in \{10\%, 50\%, 100\%\}$. The previously used method from the literature, namely \texttt{SOCP} by \cite{Gaar-et-al:2024}, is inapplicable due to non-convexity of $q(y)$, which requires \texttt{MIX++} to rely on McCormick linearizations.
The average results are presented in Table~\ref{tab:BNQKP}.
Overall, \texttt{FO} outperforms \texttt{MIX++}, solving 109 out of 180 instances with an average runtime of $279.3$ seconds and a gap of $13.0\%$, compared to $100$ solved instances, $351.5$ seconds runtime, and an $18.8\%$ gap for \texttt{MIX++}. The follower problem matrix density $p$ highlights a clear structural trade-off. At low density ($p = 10\%$), \texttt{MIX++} performs better ($51$ vs.\ $38$ solved) because sparse interaction matrices keep McCormick linearizations compact, which also explains its better overall \texttt{Gap*}. The high \texttt{RGap} of \texttt{FO} at this density stems from weak root incumbents rather than a weak root bound, as indicated by the similar \texttt{RGap*} values for both methods. However, as density increases to $50\%$ and $100\%$, the proliferation of cross-product terms causes \texttt{MIX++}'s performance to degrade sharply, with runtimes increasing to $390.0$\,s and $460.3$\,s and solved instances dropping to $27$ and~$22$, respectively. In contrast, \texttt{FO} scales far more robustly on dense instances, maintaining stable runtimes ($\sim 276$\,s--$288$\,s) and significantly lower optimality gaps ($\sim 8.5\%$). 

% Table generated by Excel2LaTeX from sheet 'summary-QINBKP'
\begin{table}[tbh]
    \setlength{\tabcolsep}{4pt}
  \centering \footnotesize
  \caption{Results over 180 BNQKP instances}
    \begin{tabular}{rrrrrrrrrrrrrrr}
     \hline
    \multicolumn{1}{l}{\texttt{Method}} & $p$ & \multicolumn{1}{c}{\texttt{nSol}} & \multicolumn{1}{c}{$t$} & \multicolumn{1}{c}{\texttt{LB}} & \multicolumn{1}{c}{\texttt{UB}} & \multicolumn{1}{c}{\texttt{rootLB}} & \multicolumn{1}{c}{\texttt{Gap}} & \multicolumn{1}{c}{\texttt{Gap*}} & \multicolumn{1}{c}{\texttt{RGap}}  & \multicolumn{1}{c}{\texttt{RGap*}} & \multicolumn{1}{c}{\texttt{nBCN}} & \multicolumn{1}{c}{$t_\Delta$} & \multicolumn{1}{c}{$t_S$} & \multicolumn{1}{c}{\texttt{$t_R$}} \\
    \hline
    \multicolumn{1}{l}{\texttt{MIX++}} & 10    & 51/60    & 204.2 & -4896.5 & -4566.4 & -6079.5 & 12.8  & 7.4   & 44.0  & 37.4  & 827.5 &   -    &   -    &  - \\
          & 50    & 27/60    & 390.0 & -5760.2 & -4735.3 & -6248.1 & 22.4  & 9.8   & 35.6  & 22.6  & 360.2 &    -   &   -    &  - \\
          & 100   & 22/60    & 460.3 & -6066.4 & -4993.5 & -6255.4 & 21.2  & 12.0  & 26.8  & 17.3  & 163.5 &  -     &    -   &  - \\
    \cline{2-15}
    &Total &	100/180	&351.5 &	-5574.4&	-4765.1&	-6194.3&	18.8&	9.7&	35.5&	25.8&	450.4&-&-&- \\
    \hline
    \multicolumn{1}{l}{\texttt{FO-RM-H-M}} & 10    & 38/60    & 272.9 & -5527.3 & -4627.0 & -6130.9 & 22.0  & 21.2  & 151.4 & 38.3  & 1289.8 & 117.8 & 139.3 & 13.2 \\
          & 50    & 37/60    & 276.4 & -5687.0 & -5191.8 & -6144.4 & 8.4   & 8.4   & 36.9  & 19.8  & 537.2 & 214.2 & 56.0  & 5.3 \\
          & 100   & 34/60    & 288.5 & -5875.3 & -5340.2 & -6155.5 & 8.5   & 8.4   & 26.3  & 14.9  & 143.5 & 263.6 & 23.2  & 1.4 \\
    \cline{2-15}
          &Total & 109/180   & 279.3 & -5696.5 & -5053.0 & -6143.6 & 13.0  & 12.7  & 71.5  & 24.3  & 656.8 & 198.5 & 72.8  & 6.6 \\
    \hline
    \end{tabular}%
  \label{tab:BNQKP}%
\end{table}%

\section{Summary and outlook}\label{sec:conclusion}

In this paper, we study bilevel optimization problems where all variables are integer, all constraints and the leader objective function are linear, and the follower objective function is a quadratic function which can be non-convex. We develop a \BC algorithm, which uses \DCs that exclude bilevel-infeasible solutions. These \DCs are based on bilevel-free sets which generalize the set introduced in \citet{watermelon,fischetti2017new} for the case with linear follower objective function. Compared to the \DC approach of \citet{Gaar-et-al:2024}, next to allowing for the non-convex case, which the bilevel-free set used in \citet{Gaar-et-al:2024} could not handle, our \DCs can be obtained by solving a cut generating LP based on an improving direction rather than a cut generating SOCP. We discuss conditions under which disjuncts are redundant and we propose and computationally compare several strategies to generate improving directions and to separate \DCs efficiently, including a pool-based heuristic scheme.

Our computational study confirms the practical value of the presented approach. On instances with a convex follower objective function, our best configurations consistently outperform the SOCP-based approach of \citet{Gaar-et-al:2024} across all five benchmark instance sets considered: On \texttt{QBCov-1} and \texttt{QBCov-2}, our best setting solves all 80 instances to optimality with an average runtime of at most 12 seconds, whereas the approach of \citet{Gaar-et-al:2024} fails to solve several instances within the time limit of 600 seconds. On the more challenging \texttt{QBMKPInt} set, our best setting solves 65 out of 99 instances, compared to only 40 for the approach of \citet{Gaar-et-al:2024}.

To assess the performance for follower problems with non-convex objective function, we introduce a new instance set \texttt{BNQKP} by modifying existing bilevel knapsack instances. Here, we compare with a state-of-the-art mixed-integer bilevel linear solver from \citet{fischetti2017new}, which can be applied to these instances after linearizing the non-convex objective function. Our best setting solves 109 of the 180 instances compared to 100 instances for the approach of \citet{fischetti2017new}.

Future work could address several directions. First, different improving directions can give rise to disjunctive cuts of different strength for the same bilevel-infeasible solution. A formal dominance analysis (in the spirit of the analysis \citet{Battista-Ralphs:2025} carry out for the mixed-integer bilevel linear case) between the cuts derived from two improving directions, and how such a result could be exploited algorithmically, e.g., to guide the selection among improving-direction generation strategies, is left for future research. Moreover, further redundancy checks and speed-ups for all parts of the separation procedure should also be an interesting topic. Finally, it can be worthwhile to investigate whether the improving-direction bilevel-free sets developed in our work extend to follower objective functions beyond the quadratic case.

\section*{Acknowledgments}

J. Lee was supported in part by U.S. ONR grant N00014-24-1-2694.
%\clearpage

\bibliographystyle{plainnat}

\bibliography{ICDC}

\end{document}